\documentclass[a4paper,11pt]{amsart}

\usepackage[matrix,arrow,tips,curve]{xy}
\usepackage[english]{babel}
\usepackage{amsmath}
\usepackage{amssymb}

\usepackage{mathrsfs}

\usepackage{enumerate}

\newtheorem{thm}{Theorem}[section]
\newtheorem{lemma}[thm]{Lemma}
\newtheorem{corollary}[thm]{Corollary}
\newtheorem{proposition}[thm]{Proposition}

\newtheorem*{thm*}{Theorem}

\theoremstyle{definition}
\newtheorem{definition}[thm]{Definition}

\newtheorem{example}[thm]{Example}
\newtheorem{remark}[thm]{Remark}
\newtheorem{parg}[thm]{}
\newtheorem{notation}[thm]{Setting}

\newcommand{\ph}{\varphi}
\newcommand{\w}{\widetilde}
\newcommand{\ma}{\mathcal}
\newcommand{\la}{\longrightarrow}

\newcommand{\pr}{\mathbb{P}}
\newcommand{\Q}{\mathbb{Q}}
\newcommand{\C}{\mathbb{C}}
\newcommand{\R}{\mathbb{R}}
\newcommand{\Z}{\mathbb{Z}}
\newcommand{\N}{\mathcal{N}_1}
\newcommand{\Nu}{\mathcal{N}^1}

\newcommand{\dom}{\operatorname{dom}}
\newcommand{\Sing}{\operatorname{Sing}}
\newcommand{\Pic}{\operatorname{Pic}}

\newcommand{\NE}{\operatorname{NE}}
\newcommand{\Exc}{\operatorname{Exc}}
\newcommand{\Supp}{\operatorname{Supp}}
\newcommand{\Lo}{\operatorname{Locus}}
\newcommand{\codim}{\operatorname{codim}}
\newcommand{\Eff}{\operatorname{Eff}}
\newcommand{\Nef}{\operatorname{Nef}}
\newcommand{\Mov}{\operatorname{Mov}}
\newcommand{\MCD}{\operatorname{MCD}}
\newcommand{\mov}{\operatorname{mov}}

\newcommand{\Bl}{\operatorname{Bl}}
\newcommand{\Id}{\operatorname{Id}}

\usepackage{etoolbox}
\patchcmd{\section}{\normalfont}{\normalfont\large}{}{}
\patchcmd{\subsection}{\bfseries}{\scshape\centering}{}{}
\patchcmd{\subsection}{-.5em}{.5em}{}{}

\title[Fano $4$-folds, flips, and blow-ups of points]{Fano $4$-folds, flips, and blow-ups of points}
\author{C.~Casagrande}
\address{Universit\`a di Torino,
Dipartimento di Matematica,
via Carlo Alberto 10,
10123 Torino - Italy}
\email{cinzia.casagrande@unito.it}
\subjclass[2010]{14J45,14J35,14E30}

\calclayout 

\begin{document}
\begin{abstract}
We study smooth, complex Fano 4-folds $X$ with large Picard number $\rho_X$, with techniques from birational geometry. Our main result is that if $X$ is isomorphic in codimension one to the blow-up of a smooth projective 4-fold $Y$ at a point, then $\rho_X\leq 12$. We give examples of such $X$ with Picard number up to $9$. The main theme (and tool) is the study of fixed prime divisors in Fano 4-folds, especially in the case $\rho_X>6$, in which we give some general results of independent interest. 
\end{abstract}
\maketitle
\renewcommand{\theequation}{\thethm}

{\footnotesize \tableofcontents}

\section{Introduction}
\noindent  
 In this paper we study (smooth, complex) Fano $4$-folds $X$ with large Picard number $\rho_X$. Let us recall that $\rho_X$ is equal to the second Betti number of $X$, and since there are finitely many families of Fano $4$-folds, $\rho_X$ is bounded. We also have an explicit bound on $\rho_X$ (see
\cite[Rem.~3.1]{cdue}), which however is, conjecturally,
 far from being sharp; the maximal value of $\rho_X$ for a Fano $4$-fold $X$ is not yet known. 

As for examples,
taking products of del Pezzo surfaces one gets Fano $4$-folds with Picard number up to $18$; to the author's knowledge, the other known examples of Fano $4$-folds have $\rho\leq 9$.
 In fact, for $\rho=7,8,9$, the author is aware of only one family (for each $\rho$) of Fano $4$-folds which is not a product of surfaces, obtained as follows. Consider the blow-up $\Bl_{p_1,\dotsc,p_r}\pr^4$ of $\pr^4$ in $r$ general points. For $r=2,\dotsc,8$ this variety is not Fano, but can be modified with a finite sequence of flips\footnote{These are $K$-positive flips, in the terminology of this paper.} in order to get a smooth Fano $4$-fold $X$ with $\rho_X=1+r\leq 9$; we refer the reader to Example \ref{example1} for more details.

The main object of this paper is the study of Fano $4$-folds obtained as in the previous example: by the blow-up of a (smooth) point, followed by a sequence of flips.

Let us first recall that Fano manifolds that can be obtained by blowing-up  a point in another manifold have been classified, in arbitrary dimension $\geq 3$, by Bonavero, Campana, and Wi\'sniewski \cite{bonwisncamp}; in particular such $X$ always has $\rho_X\leq 3$. 

In the case where we allow also flips, our main result is a bound on the Picard number. 
\begin{thm}\label{onepoint}
Let $X$ be a smooth Fano $4$-fold. Suppose that there exist a normal and $\Q$-factorial projective variety $Y$, and a smooth point $p\in Y$, such that $X$ and $\Bl_pY$ are isomorphic in codimension one. Then $\rho_X\leq 12$.
\end{thm}
As an intermediate result towards Theorem \ref{onepoint},
we also give a bound on the Picard number of Fano $4$-folds obtained, from a $4$-fold with $\rho=1$, by blow-ups of points and flips. 
\begin{thm}\label{manypoints}
 Let $X$ be a smooth Fano $4$-fold. Suppose that there exist a normal and $\Q$-factorial projective variety $Y$ with $\rho_Y=1$, and smooth distinct points $p_1,\dotsc,p_r\in Y$, such that $X$ and $\Bl_{p_1,\dotsc,p_r}Y$ are isomorphic in codimension one. Then $\rho_X\leq 9$, and if $\rho_X\geq 8$, then $Y\cong\pr^4$.
\end{thm}
\noindent The bound in Theorem \ref{manypoints} is sharp (see Example \ref{example1}).  We do not know whether the bound in Theorem \ref{onepoint} is sharp; we show some properties of $X$ when $\rho_X=12$ is maximal, see Theorem \ref{finally}.

\medskip

One of the main themes (and tools) of this work is the study of \emph{fixed prime divisors\footnote{``Fixed'' means, as usual, that $D$ is a fixed component of the linear system $|mD|$ for every $m\geq 1$.} in Fano $4$-folds}. 

Given a fixed prime divisor $D$ in a Fano variety $X$, there is always a sequence of flips $X\dasharrow X'$ such that the transform  of $D$ in $X'$ is the exceptional divisor of an elementary divisorial contraction $f\colon X'\to Y$. 
 In general, the same fixed divisor can be exceptional for different  divisorial contractions (see Examples \ref{special} and \ref{simple}).

Suppose now that $X$ is a Fano $4$-fold with $\rho_X\geq 7$.
After \cite{eff} (see Th.~\ref{typefixed})  there are only four possibilities for  the elementary divisorial contraction $f$ as above, namely: 
\begin{enumerate}[--]
\item type $(3,0)^{sm}$: $f$ is
the blow-up of a smooth point;
\item type $(3,1)^{sm}$: $f$ is the blow-up of a smooth curve;
\item type $(3,2)$: $f$ is
a divisorial contraction sending the exceptional divisor to a surface;
\item type $(3,0)^Q$: $f$ is
  a divisorial contraction with $\Exc(f)$ an irreducible quadric, with normal bundle $\ma{O}_Q(-1)$, contracted to a (mildly) singular point.
\end{enumerate}
 Moreover we show that, in a neighbourhood of the exceptional divisor, 
 $f$ is uniquely determined by $D$ (Lemma \ref{elementarydivisorial}), so that
it makes sense to define the type of a fixed prime divisor $D$ as the type of $f$.

Thus we have four possible types of fixed prime divisors in a Fano $4$-fold $X$ with $\rho_X\geq 7$, and Th.~\ref{onepoint} can be reformulated as follows: if a 
  Fano $4$-fold $X$ contains a fixed prime divisor of type $(3,0)^{sm}$, then $\rho_X\leq 12$. 

The author plans to pursue this study in order to get a (good) bound on the Picard  number of Fano $4$-folds containing a fixed prime divisor of one the remaining three types; this program, if completed, would yield a good bound on the Picard number of Fano $4$-folds, and was one of the motivations for this work.

Before getting to the proof of Th.~\ref{onepoint}, we need to show many intermediate results. In particular, in section \ref{hard}, we give some general results on the geometry of Fano $4$-folds $X$ with $\rho_X\geq 7$, which are of independent interest. 

We use mainly techniques from birational geometry and the Minimal Model Program. In particular, we rely on the fact
that Fano varieties are Mori dream spaces \cite[Cor.~1.3.2]{BCHM}, and use all the standard properties and techniques concerning Mori dream spaces \cite{hukeel}. 

We also use some previously known bounds on the Picard number $\rho_X$ of a Fano $4$-fold $X$. In particular,
we define an invariant of $X$, the \emph{Lefschetz defect}, as 
$$\delta_X:=\max\bigl\{\dim\text{coker} \bigl(i_*\colon\N(D)\to\N(X)
\bigr)\,|\,i\colon D\hookrightarrow X\text{ a prime divisor}\bigr\}.$$
If $D$ is ample, then $i_*\colon\N(D)\to\N(X)$ is an isomorphism, so that $\delta_X$ measures the failure of Lefschetz hyperplane theorem for non-ample divisors.
We refer the reader to \cite{codim,gloria} for the properties of the Lefschetz defect of Fano varieties of arbitrary dimension. When $X$ is a Fano $4$-fold which is not a product of surfaces,
we have $\delta_X\leq 3$, and if moreover
$\delta_X\geq 2$, then $\rho_X\leq 12$ (see Th.~\ref{summary}). Therefore in many situations we can reduce to the case where $\delta_X\leq 1$.

\medskip

Let us describe in more detail the content of the paper, and our strategy for the proof of Th.~\ref{onepoint}.

In section \ref{notandprel} we introduce the notation and the terminology, we recall some preliminary results, and we show some related properties needed in the sequel.

In section \ref{secmanypoints}  we prove Th.~\ref{manypoints}. 
The proof is based on the computation of the difference  $h^0(Y,-K_Y)-h^0(X,-K_X)$, using Hirzebruch-Riemann-Roch, and on a result by Hwang which gives a bound on $h^0(Y,-K_Y)$ for Fano $4$-folds with $\rho_Y=1$; see par.~\ref{introsecmanypoints} for an overview.

Then we introduce and study fixed prime divisors of type $(3,0)^{sm}$ in a Fano $4$-fold $X$ (section \ref{every1}). While the other types of fixed prime divisors are defined only when $\rho_X\geq 6$, we need to work with fixed prime divisors of type $(3,0)^{sm}$ in any Picard number, in order to be able to use an inductive strategy. In particular  
we show that, given  a fixed divisor $D\subset X$ of type $(3,0)^{sm}$,  there exists a sequence of flips $X\dasharrow X'$ and a blow-up of a point $f\colon X'\to Y$ such that $\Exc(f)$ is the transform of $D$, and $Y$ is smooth and Fano (Prop.~\ref{fanotarget}). 

In section \ref{every2} we show that if every fixed prime divisor of $X$ is of type $(3,0)^{sm}$, then $\rho_X\leq 11$ (Th.~\ref{everyfixeddiv}). This is an intermediate result needed for the proof of Th.~\ref{onepoint}; the proof is based on Th.~\ref{manypoints} and on a result from \cite{eff}, see par.~\ref{introevery2} for an overview.

The heart of the paper is section \ref{hard}, where we  make a detailed study of Fano $4$-folds with $\rho\geq 6$ (or $\rho\geq 7$), depending on their fixed prime divisors. 
 After having recalled from \cite{eff} the four types of fixed prime divisors (section \ref{tf}), we show some properties of $X$ when it contains a fixed prime divisor of type $(3,1)^{sm}$ (section \ref{a}) or $(3,0)^Q$ (section \ref{b}). Then we use the properties of fixed prime divisors to give a simple description of the dual cone of the cone of movable divisors (Lemmas \ref{movdual} and \ref{movdual2}). 
Finally we show an important property of fixed prime divisors $E$ of type $(3,2)$: if $E$ is disjoint from another prime divisor, then $\rho_X\leq 12$ (Prop.~\ref{zero}).

Building on all the previous results, we conclude section \ref{hard} with the proof of Th.\ref{onepoint}; let us give an idea of the strategy.

Let $X$ be a Fano $4$-fold containing a fixed prime divisor $D$ of type $(3,0)^{sm}$, and assume that $\rho_X\geq 12$. We know from Prop.~\ref{fanotarget} that there exists a smooth Fano $4$-fold $Y$ and a point $p\in Y$ such that $X$ and $\Bl_pY$ are isomorphic in codimension one. 

We use the results in section \ref{every} to show that $Y$ must have some fixed prime divisor $B$ not of type $(3,0)^{sm}$, thus $B$ is of type $(3,2)$, $(3,1)^{sm}$, or $(3,0)^Q$. Using the results in sections \ref{a} and \ref{b}, we show that in any case $Y$ must contain a fixed prime divisor $E_Y$ of type $(3,2)$. Finally this yields a fixed prime divisor $E_X\subset X$, of type $(3,2)$, and disjoint from $D$, which implies that $\rho_X=12$ by Prop.~\ref{zero}.

Finally we present some examples in section \ref{examples}.

\smallskip

\noindent{\bf Acknowledgements.} I would like to thank Carolina Araujo for 
her help with 
Example \ref{example1}, and St\'ephane Druel for many useful conversations and comments. 
\section{Notations and preliminaries}\label{notandprel}
\noindent In this section we set the terminology and the notation, about the Minimal Model Program, Mori dream spaces, and families of curves. 
We also recall several preliminary results that we need on Fano $4$-folds, mostly from \cite{eff}, and show some related properties that will be used in the sequel.
\subsection{Preliminaries from the MMP and on Mori dream spaces}\label{notMMP}
\noindent We work over the field of complex numbers.
We refer the reader to \cite{kollarmori} for the terminology related to the Minimal Model Program, to \cite{hukeel} for the definition and properties of Mori dream spaces, and to \cite[\S 2]{eff} for an overview of most of the properties of Mori dream spaces that will be used in this paper.

 Let $X$ be a normal, $\Q$-factorial,  projective Mori dream space.

If $f\colon X\dasharrow Y$ is a rational map, the domain $\dom(f)$ of $f$ is the largest
open subset of $X$ where $f$ is regular.

A {\bf contraction} 
of $X$ is a morphism 
with connected fibers $f\colon X\to Y$
onto a normal projective variety. The exceptional locus $\Exc(f)$ is the set of points of $X$
where $f$ is not an isomorphism. 
If $D$ is a divisor in $X$, we say that $f$ is
{\bf $D$-positive} (respectively {\bf $D$-negative}) 
if $D\cdot C>0$ (respectively
$D\cdot C<0$) for every curve $C$ in $X$ such that
$f(C)=\{pt\}$. When $D=K_X$, we just say {\bf $K$-positive} (or 
{\bf $K$-negative}).

The contraction $f$ is {\bf elementary} if $\rho_X-\rho_Y=1$; as usual, $f$ can be divisorial, small, or of fiber type. When $f$ is elementary, we say that  $f$ is {\bf of type $\mathbf{(a,b)}$} if $\dim\Exc(f)=a$ and
$\dim f(\Exc(f))=b$. We will also use the following terminology.
\begin{definition}\label{typecontraction}
Let $f\colon X\to Y$ be an elementary divisorial contraction. 

We say that $f$ is
 {\bf of type $\mathbf{(n-1,m)^{sm}}$} if $f$ is the blow-up of a smooth, connected, $m$-dimensional subvariety of $Y$, contained in the smooth locus of $Y$. 

We say that $f$ is  {\bf of type $\mathbf{(n-1,0)^Q}$} if $\Exc(f)\cong Q$, $Q\subset\pr^n$ a quadric, with normal bundle $\ma{N}_{\Exc(f)/X}\cong\ma{O}_Q(-1)$, and $f(\Exc(f))$ is a point.
\end{definition}
\noindent Notice that if $f$ is of type $(n-1,m)^{sm}$, then $X$ is smooth if and only if $Y$ is smooth.

We denote by $\mathcal{N}^1(X)$ (respectively $\mathcal{N}_1(X)$) 
 the $\R$-vector space of Cartier divisors (respectively one-cycles)
with
real coefficients, modulo numerical equivalence;
$[D]$ is the numerical equivalence class in $\mathcal{N}^1(X)$ 
 of a
divisor $D$ in $X$, and similarly $[C]\in\N(X)$ for a curve $C\subset X$;
$\equiv$ stands for numerical equivalence.

For any closed subset $Z\subseteq X$, we set
$$\N(Z,X):=i_*(\N(Z))\subseteq\N(X),$$
where $i\colon Z\hookrightarrow X$ is the inclusion. 

For any subset $H\subseteq \N(X)$, we set $H^{\perp}:=
\{\gamma\in\Nu(X)\,|\,h\cdot\gamma=0\text{ for every }h\in H\}$, and
similarly for $H\subseteq \Nu(X)$.
For any divisor $D$ in $X$,
$D^{\perp}:=[D]^{\perp}\subseteq \N(X)$, and similarly for a curve $C\subset X$.

An effective divisor $D$ in $X$ is {\bf movable} if its stable base locus has
codimension at least $2$. A prime divisor is {\bf fixed} if it not movable.

We will consider the following convex polyhedral cones:
$$\Nef(X)\subseteq\Mov(X)\subseteq\Eff(X)\subset\Nu(X)\quad\text{and}\quad
\mov(X)\subseteq\NE(X)\subset\N(X),
$$
where:\\
$\Nef(X)$ is the cone of nef classes,\\
$\Mov(X)$ is the convex cone generated by classes of
movable divisors,\\
$\Eff(X)$ is the convex cone generated by classes of
effective divisors,\\
$\NE(X)$ is the convex cone generated by classes of
effective curves,\\
$\mov(X)$ is the cone of ``moving curves'', namely the convex cone generated by classes of curves moving in a family of curves covering $X$.

An {\bf extremal ray} of $X$ is a one-dimensional
face of $\NE(X)$.
We use greek letters $\sigma,\tau,\eta,$ etc.\ to denote convex
polyhedral cones and their faces in $\N(X)$ or $\Nu(X)$, with the exception of extremal rays of $\NE(X)$, which are denoted with the more customary notation $R$.
We denote by $\langle x_1,\dotsc,x_r \rangle$ the convex cone generated by $x_1,\dotsc,x_r$ in $\N(X)$ or $\Nu(X)$.
If $\sigma$ is a convex polyhedral cone in  $\N(X)$ (respectively in $\Nu(X)$), we denote by $\sigma^{\vee}\subset\Nu(X)$ (respectively $\sigma^{\vee}\subset\N(X)$) its dual cone. 

We have the relations:
$$\NE(X)=\Nef(X)^{\vee}\qquad\text{and}\qquad \mov(X)=\Eff(X)^{\vee}$$
(see \cite[\S 11.4.C and references therein]{lazII}, \cite{carolina}, and \cite[\S 2.16]{eff}).

Let $f\colon X\to Y$ be a contraction. We consider the push-forward of one-cycles $f_*\colon\N(X)\to\N(Y)$, and 
set $\NE(f):=\NE(X)\cap\ker f_*$. We also say that $f$ is the
contraction of $\NE(f)$.
 
Consider an elementary contraction $f\colon X\to Y$ and the extremal
ray $R:=\NE(f)$. We say that $R$ is birational, divisorial,
 small, of type $(a,b)$, or of type $(n-1,m)^{sm}$, 
if $f$ is. We set $\Lo(R):=\Exc(f)$, namely $\Lo(R)$ is the
union of the curves whose class belongs to $R$. 
If $D$ is a
divisor in $X$, we say that
$D\cdot R>0$ if $D\cdot C>0$ for a curve $C$ with $[C]\in R$,
equivalently if $f$ is $D$-positive; similarly for $D\cdot R=0$ and
$D\cdot R<0$.

We refer the reader to \cite{kollarmori} for the definition of the {\bf flip} $g\colon X\dasharrow\w{X}$
of a small elementary elementary contraction $f\colon X\to Y$. We also say that $g$ is the flip of the
small extremal ray $\NE(f)$.  If $D$ is a divisor on $X$ such that $f$ is
$D$-negative (respectively, $D$-positive), we say that $g$ is a 
 {$D$-negative
flip} (respectively, 
$D$-positive flip). Finally, 
 when $D=K_X$, we just say $K$-positive or $K$-negative.

A {\bf MMP for a divisor $D$} in $X$ is a sequence
$$X=X_1\stackrel{\ph_1}\dasharrow X_2\dasharrow\cdots\dasharrow X_{k-1}\stackrel{\ph_{k-1}}{\dasharrow} X_k$$
where every $\ph_i$ is either a $D_i$-negative flip (and $D_{i+1}$ is the transform of $D_i$), or a $D_i$-negative elementary divisorial contraction (and $D_{i+1}=(\ph_i)_*D_i$), and moreover: if $[D]\in\Eff(X)$, then $D_k$ is nef; otherwise $X_k$ has a $D_k$-negative elementary contraction of fiber type.
Let us consider the composition $\ph\colon X\dasharrow X_k$, let $C\subset X_k$ be an irreducible curve such that $C\cap\dom\ph^{-1}\neq\emptyset$, and let $C_X\subset X$ be its transform. Then $D\cdot C_X\geq D_k\cdot C$ (this follows from negativity of contractions, see \cite[Lemma 3.38]{kollarmori}).

A {\bf small $\Q$-factorial modification (SQM)} of $X$ is a normal and
$\Q$-factorial projective variety $\w{X}$, together with a
birational map $\ph\colon X\dasharrow \w{X}$ which is an isomorphism
in codimension $1$. 
Since $X$ is a Mori dream space, $\ph$ can always be factored as a finite sequence of flips.

A {\bf rational contraction}\footnote{In \cite{hukeel} this is also called a ``contracting rational map''.} of $X$  is a rational map $f\colon X\dasharrow Y$ that can be factored as in the following diagram:
\stepcounter{thm}
\begin{equation}\label{f}
\xymatrix{{X}\ar@{-->}[r]_{\ph}\ar@{-->}@/^1pc/[rr]^f&{\w{X}}\ar[r]_{\tilde{f}}&{Y,}
}
\end{equation}
where $\ph$ is a SQM and $\tilde{f}$ is a contraction.
We say that $f$ is {\bf elementary} if $\rho_X-\rho_Y=1$, equivalently if $\tilde{f}$ is elementary. We say that $f$ is {\bf of fiber type} if $\dim Y<\dim X$. 
When $f$ is birational, 
we denote by $\Exc(f)$ the closure in $X$ of the exceptional locus of $f\colon\dom(f)\to Y$. As in the regular case, when $f$ is elementary and birational, $f$ can be {\bf divisorial} (if $\Exc(f)$ is a prime divisor) or {\bf small} (if $\codim\Exc(f)\geq 2$). 

Let $f\colon X\dasharrow Y$ be a rational contraction such that $Y$ is $\Q$-factorial. Then $Y$ is again a Mori dream space, and for every rational contraction $g\colon Y\dasharrow Z$, the composition $g\circ f\colon X\dasharrow Z$ is again a rational contraction, see for instance \cite[Rem.~2.8]{eff}.\label{oggi} These properties will be used throughout the paper.
\begin{remark}[\cite{eff}, Rem.~3.7]\label{Kneg}
Let $X$ be a smooth Fano $4$-fold and $f\colon X\dasharrow Y$ a rational contraction. Then there exists a factorization of $f$ as \eqref{f}, with moreover $\tilde{f}$ $K$-negative.
\end{remark}

 We denote by  $\text{\bf MCD}\mathbf{(X)}$ the collection of the cones $f^*(\Nef(Y))$ for all rational contractions  $f\colon X\dasharrow Y$. This is a fan in $\Nu(X)$, supported on the movable cone $\Mov(X)$, called the 
{\bf Mori chamber decomposition} of $\Mov(X)$. The
 cones in $\MCD(X)$ are in bijection with rational contractions of $X$.

The following result will be very useful.
\begin{lemma}[\cite{eff}, Lemma 2.21 and Prop.~2.22]\label{movf}
Let $X$ be a projective, normal, and $\Q$-factorial Mori dream space, $f\colon X\dasharrow Y$ a rational contraction, and $\sigma$ the minimal face of $\Eff(X)$ containing $f^*(\Nef(Y))$.
Consider a factorization of $f$ as \eqref{f},
and let $F\subset \w{X}$ be a general fiber of $\tilde{f}$.
Then 
$$\rho_F\geq \dim\N(F,\w{X})=\rho_X-\dim\sigma=\dim\bigl(\ker\tilde{f}_*\cap \mov(X)\bigr).$$ 
\end{lemma}
\subsection{Families of curves}\label{families}
\noindent We refer the reader to \cite{kollar} for generalities on families of curves.

Let $X$ be a smooth projective variety. 
By a family of  curves in $X$ we mean an irreducible closed subset $V\subset\text{Chow}(X)$ such that for general $v\in V$, the corresponding cycle $C_v$ of $X$ is an integral curve; then  for every $v\in V$ the cycle $C_v$ is a connected curve in $X$. The numerical equivalence class $[C_v]\in\N(X)$ does not depend on
$v\in V$, and we denote it by $[V]$.
The anticanonical degree of $V$ is $-K_X\cdot[V]$. We say that $V$ is a family of rational curves if for general $v\in V$ the curve $C_v$ is rational.
We say that $V$ is maximal if it is an irreducible component of $\text{Chow}(X)$.

We denote by $\Lo(V)$ the union of the curves $C_v$ for $v\in V$, and
we say that $V$ is covering if $\Lo(V)=X$. 
 If $Z\subset X$ is a closed subset, we set $V_Z:=\{v\in V\,|\,C_v\cap Z\neq\emptyset\}$; this is a
closed subset of $V$. 
We will need the following.
\begin{lemma}[\cite{kollar}, Prop.~II.4.19, see also \cite{occhettaGM}, Lemma 4.1] \label{kol}
Let $V$ be a family of rational curves, and $Z\subseteq X$ a closed subset such that  $C_v$ is an integral curve for every $v\in V_Z$.
Then $\N(\Lo(V_Z),X)=\N(Z,X)+\R[V]$.
\end{lemma}
 \subsection{Flips and SQM's of a smooth $4$-fold}\label{flips}
\noindent Let $X$ be a smooth projective $4$-fold.
An {\bf exceptional plane} in $X$ is a closed subset $L\subset X$
such that $L\cong\pr^2$ and
$\mathcal{N}_{L/X}\cong\mathcal{O}_{\pr^2}(-1)^{\oplus 2}$. 
We denote by
$C_L\subset L$ a line; notice that $-K_X\cdot C_L=1$.
An
{\bf exceptional line} in $X$ is a curve $\ell\cong\pr^1$ with
$\mathcal{N}_{\ell/X}\cong\mathcal{O}_{\pr^1}(-1)^{\oplus 3}$; notice
that $K_X\cdot \ell=1$.
\begin{thm}[\cite{kawsmall}]\label{lorenzo}
Let $X$ be a smooth projective $4$-fold and $R$ a small, $K$-negative extremal ray of $\overline{\NE}(X)$. Then $\Lo(R)$ is a disjoint union of exceptional planes.
\end{thm}
\begin{thm}[\cite{kawsmall} and \cite{eff}, Rem.~3.6]\label{SQM0}
Let $X$ be a smooth Fano $4$-fold and $\ph\colon
X\dasharrow \w{X}$ a SQM. 
 Then 
$\w{X}$ is smooth,
$X\smallsetminus\dom(\ph)$ is a disjoint union of
exceptional planes $L_1,\dotsc,L_s\subset X$, and  $\w{X}\smallsetminus\dom(\ph^{-1})$
is a disjoint union of exceptional lines $\ell_1,\dotsc,\ell_s\subset\w{X}$; moreover
$\ph$ factors as 
$$
\xymatrix{
&\widehat{X}\ar[dl]_f\ar[dr]^g&\\
X\ar@{-->}[rr]^{\ph}&&{\w{X}}
}$$
where $f$ is the blow-up of  $L_1\cup\cdots\cup L_s$,
 and $g$ is the blow-up of  $\ell_1\cup\cdots\cup \ell_s$.
\end{thm}
\begin{lemma}[\cite{eff}, Rem.~3.6]\label{SQM}
Let $X$ be a smooth Fano $4$-fold and $\ph\colon
X\dasharrow \w{X}$ a SQM.
\begin{enumerate}[(1)]
\item 
 If ${C}\subset\w{X}$ is an irreducible curve such that
$C\cap\dom(\ph^{-1})\neq\emptyset$, 
 $n$ is the number of points of ${C}$ 
belonging to an exceptional line, and $C_X\subset X$ is the transform of $C$,
we have
$$-K_{\w{X}}\cdot C\geq -K_X\cdot C_X+n\geq 1+n\geq 1.$$
\item If ${C}\subset\w{X}$ is an irreducible curve such that 
  $-K_{\w{X}}\cdot C\leq 0$, then $C$
  is an exceptional line and  $-K_{\w{X}}\cdot C=-1$.
\item Let $T\subseteq\w{X}$ be the union of integral curves of anticanonical degree $1$. Then $T\subset\dom(\ph^{-1})$, and $T$ cannot intersect exceptional lines.
\end{enumerate}
\end{lemma}
\begin{remark}\label{SQM2}
Let $X$ be a smooth Fano $4$-fold and $\ph\colon
X\dasharrow \w{X}$ a SQM. If $L\subset X$ is an exceptional plane, then either $L\cap\dom(\ph)=\emptyset$, or $L\smallsetminus\dom(\ph)$ is a finite set (possibly empty).
\end{remark}
\begin{proof}
By Th.~\ref{SQM0}, the indeterminacy locus of $\ph$ is a disjoint union of  exceptional planes $L_1,\dotsc,L_s$.
Let $\w{B}$ be an ample divisor in $\w{X}$, and $B$ its transform in $X$, so that $\ph\colon X\dasharrow\w{X}$ is a MMP for the divisor $B$.
We have $B\cdot C_{L_i}<0$ for every $i=1,\dotsc,s$.

If $L\subset X$ is an exceptional plane different from $L_1,\dotsc,L_s$, then we can choose a line $C_L\subset L$ such that $C_L\cap\dom(\ph)\neq\emptyset$. If $\w{C}_L\subset\w{X}$ is the transform of $C_L$, we have $B\cdot C_L\geq\w{B}\cdot\w{C}_L >0$, thus the class $[C_L]$ cannot be proportional to $[C_{L_i}]$ for every $i=1,\dotsc,s$. This implies that $\dim(L\cap L_i)\leq 0$ for every $i=1,\dotsc,s$.
\end{proof}
\begin{remark}\label{movable}
Let $X$ be a smooth Fano $4$-fold, $C\subset X$ an irreducible curve, and $B$ a movable divisor in $X$ such that  $B\cdot C<0$. 
Then $C$ is contained in an exceptional plane.
\end{remark}
\begin{proof}
Since $B$ is movable,
by running a MMP for $B$ we get a SQM $\ph\colon X\dasharrow \w{X}$ such that the transform $\w{B}$ of $B$ in $\w{X}$ is nef. By Th.~\ref{SQM0}, the indeterminacy locus of $\ph$ is a disjoint union of finitely many exceptional planes. If $C$ intersects $\dom\ph$, let $\w{C}\subset \w{X}$ be its transform. Then we have $\w{B}\cdot \w{C}\leq B\cdot C<0$, a contradiction. Thus $C$ is contained in an exceptional plane.
\end{proof}
\begin{lemma}\label{dimensions}
Let $X$ be a smooth Fano $4$-fold and $\ph\colon X\dasharrow\w{X}$ a SQM. Let $Z\subseteq X$ and $\w{Z}\subseteq\w{X}$ be closed subsets such that $Z$ contains the indeterminacy locus of $\ph$, $\w{Z}$ contains the indeterminacy locus of $\ph^{-1}$, and $\ph(Z\cap\dom\ph)=\w{Z}\cap\dom(\ph^{-1})$.
Then $\dim\N(Z,X)=\dim\N(\w{Z},\w{X})$.
\end{lemma}
\begin{proof}
Consider the factorization $X\stackrel{f}{\leftarrow}\widehat{X}\stackrel{g}{\to}\w{X}$ of $\ph$ as in Th.~\ref{SQM0}.
 By our assumptions we have $f^{-1}(Z)=g^{-1}(\w{Z})$; let us denote this closed subset by $\widehat{Z}$. Then $\widehat{Z}$ contains all exceptional divisors of $f$ and $g$, so that $\N(\widehat{Z},\widehat{X})$ contains both $\ker f_*$ and $\ker g_*$. Moreover $f(\widehat{Z})=Z$ and $g(\widehat{Z})=\w{Z}$, hence
$$f_*(\N(\widehat{Z},\widehat{X}))=\N(Z,X)\quad\text{and}\quad
g_*(\N(\widehat{Z},\widehat{X}))=\N(\w{Z},\w{X}).$$
Since $\dim\ker f_*=\dim\ker g_*$, we get the statement.
\end{proof}
\subsection{Preliminary results on Fano $4$-folds}
\noindent Let $X$ be a smooth Fano variety. Recall from the Introduction that the {\bf Lefschetz defect} of $X$ is
$$\delta_X:=\max\{\codim\N(D,X)\,|\,D\subset X\text{ a prime divisor}\};$$
we will use the following bounds on $\delta_X$ and $\rho_X$.
\begin{thm}[\cite{eff}, Th.~3.11 and \cite{cdue}, Th.~1.2]\label{summary}
Let $X$ be a smooth Fano $4$-fold which is not 
 a product of surfaces. Then $\delta_X\leq 3$, and moreover: if $\delta_X=3$, then $\rho_X\leq 6$; if $\delta_X=2$, then $\rho_X\leq 12$.
\end{thm}
\begin{thm}[\cite{eff}, Th.~1.1] \label{elemft}
Let $X$ be a smooth Fano $4$-fold. If $X$ has an elementary rational contraction of fiber type, then $\rho_X\leq 11$.
\end{thm}
\begin{parg}\label{basic0}
Let $X$ be a smooth projective $4$-fold, and $f\colon X\to Y$ an elementary, $K$-negative divisorial contraction of type $(3,2)$; set $D:=\Exc(f)$. It is well-known that $f$ has at most finitely many $2$-dimensional fibers, and that every $1$-dimensional fiber $F$ of $f$ satisfies $F\cong\pr^1$, $D\cdot F=-1$, and $-K_X\cdot F=1$. In particular, $D$ is covered by a family of rational curves of anticanonical degree $1$. Moreover, if $f$ has no $2$-dimensional fibers, then it is of type $(3,2)^{sm}$ \cite[Th.~4.1(1)]{AWaview}.
\end{parg}
\begin{remark}\label{basic}
Let $X$ be a smooth projective $4$-fold, and $f\colon X\to Y$ an elementary, $K$-negative divisorial contraction; set $D:=\Exc(f)$.
\begin{enumerate}[(1)]
\item
If $f$ is of type $(3,0)$, then $\dim\N(D,X)=1$.
\item
If $f$ is of type $(3,1)$, then $\dim\N(D,X)=2$. 
\item
If $D$ contains  an irreducible surface $T$ such that  $\dim\N(T,X)=1$ and $f$ is finite on $T$, then $f$ is
of type $(3,2)$ and  $\dim\N(D,X)=2$.
\item
If $f$ is of type $(3,2)$, and $D$ contains an exceptional plane $L$, then $f$ is finite on $L$ and $\dim\N(D,X)=2$.
\end{enumerate}
\end{remark}
\begin{proof}
We have $\dim\ker f_*=1$ and $f_*(\N(D,X))=\N(f(D),Y)$; this easily yields (1) and (2). For (3), $\dim f(D)\geq\dim f(T)=2$, thus $f$ is of type $(3,2)$, $f(D)=f(T)$, and
 $f_*(\N(T,X))=f_*(\N(D,X))$. Finally, if $f$ is of type $(3,2)$, an exceptional plane cannot be contained in a fiber of $f$ by \cite[Th.~4.7]{AWaview}, thus (4) follows from (3).
\end{proof}
\begin{lemma}\label{torino}
Let $X$ be a smooth Fano $4$-fold, and $D\subset X$ a fixed prime divisor such that every $D$-negative extremal ray of $\NE(X)$ is of type $(3,2)$. Then the following holds:
\begin{enumerate}[(1)]
\item $-K_X+D$ is nef, and $(-K_X+D)^{\perp}\cap\NE(X)$ is a face containing all $D$-negative extremal rays of $\NE(X)$ (and no others);
\item either $\NE(X)$ has a unique $D$-negative extremal ray, or $\rho_X\leq 6$ and $\dim\N(D,X)\leq 3$.
\end{enumerate}
\end{lemma}
\begin{proof}
This proof is very similar to that of \cite[Rem.~5.3]{eff}.
Let $R_1,\dotsc,R_n$ be the $D$-negative extremal rays of $\NE(X)$ (notice that $n\geq 1$, because $D$ is not nef). By assumption, each $R_i$ is of type $(3,2)$; if $F_i\subset D$ is a general fiber of the contraction of $R_i$, we have $(-K_X+D)\cdot F_i=0$ (see \ref{basic0}).

On the other hand, if $R$ is an extremal ray of $\NE(X)$ different from $R_1,\dotsc,R_n$, then $D\cdot R\geq 0$, $-K_X\cdot R>0$, and hence  $(-K_X+D)\cdot R>0$. This shows (1).

Suppose now that $n>1$. In particular this implies that $X$ is not a product of surfaces, and hence $\delta_X\leq 3$ by Th.~\ref{summary}. 
Since $n>1$, the face $(-K_X+D)^{\perp}\cap\NE(X)$ has dimension at least $2$; let $\tau$ be one of its $2$-dimensional faces. The contraction $g\colon X\to Z$ with $\NE(g)=\tau$ is such that $\rho_X-\rho_Z=2$ and $\dim g(D)\leq 1$.
We have $g_*(\N(D,X))=\N(g(D),Z)$,
 therefore $\dim g_*(\N(D,X))\leq 1$ and hence $\dim\N(D,X)\leq 3$.
Since $\delta_X\leq 3$, we get $\rho_X\leq 6$. This shows (2).
\end{proof}
\begin{remark}\label{pisa}
Let $X$ be a smooth Fano $4$-fold with $\rho_X\geq 6$. 
\begin{enumerate}[(1)]
\item
Every divisorial extremal ray of $\NE(X)$ is of type $(3,2)$.
\item
Let $D\subset X$ be the locus of an extremal ray of type $(3,2)$.
 Then $D$ does not contain exceptional planes, and every $D$-negative extremal ray of $\NE(X)$ is of type $(3,2)$;
\item either $\NE(X)$ has a unique $D$-negative extremal ray, or $\rho_X=6$, $\delta_X=3$, and $\dim\N(D,X)=3$.
\end{enumerate}
\end{remark}
\begin{proof}
The statement is clear if $X$ is a product of surfaces. Otherwise, $\delta_X\leq 3$ by Th.~\ref{summary}, hence $\dim\N(E,X)\geq\rho_X-3\geq 3$ for every prime divisor $E\subset X$. Thus (1) follows from  Rem.~\ref{basic}(1)-(2), (2) follows from Rem.~\ref{basic}(4) and (1), and (3) follows from (2) and Lemma \ref{torino}(2).
 \end{proof}
\begin{lemma}\label{divisorcovered}
Let $X$ be a smooth Fano $4$-fold with $\rho_X\geq 6$, and $D\subset X$ a prime divisor covered by a family of rational curves of anticanonical degree $1$. Then one of the following holds:
\begin{enumerate}[$(i)$]
\item $D$ is nef;
\item $D$ is the locus of an extremal ray of type $(3,2)$.
\end{enumerate}
\end{lemma}
\begin{proof}
Let us assume that $(i)$ does not hold; then there exists an extremal ray $R$ of $\NE(X)$ such that $D\cdot R<0$.
If $R$ is divisorial, then it is of type $(3,2)$ because $\rho_X\geq 6$ (see Rem.~\ref{pisa}(1)), and we have $(ii)$.
We show that $R$ cannot be small. By contradiction,
suppose otherwise, and let $L\subseteq\Lo(R)\subset D$ be an exceptional plane (see Lemma \ref{lorenzo}). 

There is a family $V$ of rational curves in $X$ with $\Lo(V)=D$ and $-K_X\cdot[V]=1$. Since $X$ is Fano, for every $v\in V$ the curve $C_v$ must be integral, thus by  Lemma \ref{kol} we have $\N(\Lo(V_L),X)=\R[V]+\N(L,X)=\R[V]+\R R$.

On the other hand $[V]\not\in R$ (for otherwise $D\subseteq\Lo(R)$), thus for every $v\in V$, the curve $C_v$ cannot be contained in $L$. This means that  $\dim\Lo(V_L)\geq 3$, thus $\Lo(V_L)=D$ and $\dim\N(D,X)=2$; since $\rho_X\geq 6$, we get $\delta_X\geq 4$. By Th.~\ref{summary} this implies that $X$ is a product of  surfaces, hence $X$ has no small extremal rays, and we get a contradiction.
\end{proof}
We will also need the following property, which
 should be well-known to experts; we include a proof for lack of reference.
\begin{lemma}\label{quadricsing}
Let $X$ be a smooth projective $4$-fold and $f\colon X\to Y$ an elementary divisorial contraction of type $(3,0)^Q$ (see Def.~\ref{typecontraction}).
 Set  $E:=\Exc(f)$ and $p:=f(E)\in Y$.
Then $p$ is an isolated terminal and factorial singularity in $Y$, $f$ is the blow-up of the ideal of $p$ in $Y$, and two cases can occur:
\begin{enumerate}[$(i)$]
\item $Q$ is a smooth quadric and $p$ is an ordinary double point;
\item $Q$ is a cone over a smooth $2$-dimensional quadric, and $p$ is locally analytically a singularity of type $\{x_0^2+x_1^2+x_2^2+x_3^2+x_4^3=0\}\subset\C^5$.
\end{enumerate}
\end{lemma}
\begin{proof}
The proof is analogous to the proof of the $3$-dimensional case in \cite[proof of (3.4.3), pp.\ 164-165]{moriannals}. Notice that $Q$ is irreducible, because $f$ is elementary by assumption. Moreover it is easy to check that $f$ is $K$-negative and that $E\cdot C=-1$ for a line $C\subset E$, hence $p$ is a terminal and factorial singularity by standard arguments.

It follows from \cite[Prop.~5.4 and 5.5]{AW} that $f$ is the blow-up of the ideal $\mathfrak{m}_{p}$ of $p$ in $Y$ and that the graded ring  $\text{gr}(\ma{O}_{Y,p}):=\oplus_{j\geq 0}(\mathfrak{m}_{p}^j/\mathfrak{m}_{p}^{j+1})$ associated to the local ring $\ma{O}_{Y,p}$ is isomorphic as a graded $\C$-algebra to $\C[x_0,x_1,x_2,x_3,x_4]/(F)$, where $F$ is the equation of $Q\subset\pr^4$.

Let us consider the complete local ring $\hat{\ma{O}}_{Y,p}$. We have 
$\text{gr}(\hat{\ma{O}}_{Y,p})\cong\text{gr}(\ma{O}_{Y,p})$ as graded $\C$-algebras, which yields $\hat{\ma{O}}_{Y,p}\cong\C[[x_0,x_1,x_2,x_3,x_4]]/(F+G)$, where $G\in\hat{\mathfrak{m}}_p^3$.

If $Q$ is smooth, with suitable changes of coordinates we get 
$\hat{\ma{O}}_{Y,p}\cong\C[[x_0,x_1,x_2,x_3,x_4]]/(x_0^2+x_1^2+x_2^2+x_3^2+x_4^2)$, hence $p$ is an ordinary double point.  If $Q$ is a cone over a smooth $2$-dimensional quadric, we get  $\hat{\ma{O}}_{Y,p}\cong\C[[x_0,x_1,x_2,x_3,x_4]]/(x_0^2+x_1^2+x_2^2+x_3^2+x_4^n)$, where $n\geq 3$. Finally if $Q$ is a join of a line and a smooth conic,  we get $\hat{\ma{O}}_{Y,p}\cong\C[[x_0,x_1,x_2,x_3,x_4]]/(x_0^2+x_1^2+x_2^2+H(x_3,x_4))$,
where $H\in\C[[x_3,x_4]]$ and $H\in\hat{\mathfrak{m}}_p^3$.
Now by imposing that the blow-up of $Y$ at $p$ is smooth, we see by a local computation that when $Q$ is singular the only case that can occur is $(ii)$.
\end{proof}
\subsection{Faces of the effective cone}\label{extreff}
\begin{parg}[Fixed prime divisors and faces of $\Eff(X)$]\label{otite}
Let $X$ be a projective, normal, and $\Q$-factorial Mori dream space. 
There is a bijection between:
\begin{enumerate}[$\bullet$]
\item the set of fixed prime divisors in $X$, and
\item the set of one-dimensional faces of $\Eff(X)$ not contained in $\Mov(X)$,
\end{enumerate}
obtained by $D\mapsto \langle[D]\rangle$ (see  \cite[\S 2.18]{eff}).
\end{parg}
\begin{lemma}\label{bijection}
Let $X$ be a projective, normal, and $\Q$-factorial Mori dream space, $D\subset X$ a fixed prime divisor, and $f\colon X\dasharrow Y$ an elementary divisorial rational contraction with $\Exc(f)=D$. 
\begin{enumerate}[(1)]
\item There is  a bijection between $2$-dimensional faces $\sigma$ of $\Eff(X)$ containing $[D]$, and one-dimensional faces of $\Eff(Y)$, given by $$\sigma\mapsto f_*\sigma;$$
\item  $\sigma\cap\Mov(X)=\{0\}$ if and only if 
$f_*\sigma\not\subset\Mov(Y)$ ($\sigma$ as in (1));
\item  there is a bijection between fixed prime divisors $B_Y\subset Y$ and fixed prime divisors $B_X\subset X$
such that $\langle [D], [B_X]\rangle$ is a $2$-dimensional face of $\Eff(X)$ and $\langle [D], [B_X]\rangle\cap\Mov(X)=\{0\}$; here $B_X$ is the 
 transform of $B_Y$ in $X$.
\end{enumerate}
\end{lemma}
\begin{proof}
By considering a SQM $\ph\colon X\dasharrow \w{X}$ such that $f\circ\ph^{-1}\colon\w{X}\to Y$ is regular, we get a push-forward map $$f_*\colon\Nu(X)\to\Nu(Y).$$
This is a surjective linear map with $\ker f_*=\R[D]$. Moreover $f_*(\Eff(X))=\Eff(Y)$, and $\ker f_*\cap\Eff(X)=\langle[D]\rangle$ is a one-dimensional face of $\Eff(X)$. This yields (1).

Let $\sigma$ be a $2$-dimensional face of $\Eff(X)$ containing $[D]$.
Since $f_*(\Mov(X))=\Mov(Y)$, 
if  $\sigma\cap\Mov(X)\neq\{0\}$, then $f_*\sigma\subset\Mov(Y)$. 

Suppose
 instead that
 $\sigma\cap\Mov(X)=\{0\}$. Then there exists a fixed prime divisor $B\subset X$ such that $\sigma=\langle[D],[B]\rangle$,
and every 
  effective divisor of $X$ with class in $\sigma$ has the form $aD+bB$ with $a,b\in\Z_{\geq 0}$. Thus every effective divisor with class in $f_*\sigma$ is a multiple of $f_*B$. This shows that $f_*B$ is not movable, and we have (2).

Finally, (3) follows from (1) and (2).
\end{proof}
The following result allows to bound $\rho_X$ in terms of how $\Eff(X)$ intersects $\Mov(X)$.
\begin{lemma}\label{useful}
Let $X$ be a smooth Fano $4$-fold and let $d$ be the minimal dimension of a face $\sigma$ of $\Eff(X)$ such that $\sigma\cap\Mov(X)\neq\{0\}$. Then $\rho_X\leq 10+d$.
\end{lemma}
\begin{proof}
Set $\tau:=\sigma\cap\Mov(X)$; then $\tau$ is a face of $\Mov(X)$, with $\dim\tau>0$,
 and $\sigma$ is the smallest face of $\Eff(X)$ containing $\tau$.
Let us choose $\tau_0\in\MCD(X)$, the Mori chamber decomposition of $\Mov(X)$, such that $\tau_0\subseteq\tau$ and $\dim\tau_0=\dim\tau$.
Then $\tau_0$ determines a rational contraction $f\colon X\dasharrow Y$ such that $f^*(\Nef(Y))=\tau_0$; in particular $\rho_Y=\dim\tau_0>0$, so that $Y$ is not a point. 
We can factor $f$ as
$$\xymatrix{{X}\ar@{-->}[r]\ar@{-->}@/^1pc/[rr]^f&{\w{X}}\ar[r]_{\tilde{f}}&{Y,}
}$$
where $X\dasharrow \w{X}$ is a SQM and $\tilde{f}$ is a $K$-negative contraction (see Rem.~\ref{Kneg}). 
Let $F\subset \w{X}$ be a general fiber of $\tilde{f}$. Since $\w{X}$ is smooth (by Th.~\ref{SQM0}) and $\tilde{f}$ is $K$-negative, $F$ is a smooth Fano variety with $\dim F\leq 3$, and $\rho_F\leq 10$ \cite[\S 12.6]{fanoEMS}.
On the other hand, by Lemma \ref{movf} we have 
$$
\rho_F\geq\dim\N(F,\w{X})=\rho_X-\dim\sigma=\rho_X-d,$$
namely $\rho_X\leq\rho_F+d\leq 10+d$.
\end{proof}
\noindent The bound in Lemma \ref{useful} can sometimes be refined in specific cases, like the following.
\begin{lemma}\label{macau}
Let $X$ be a smooth Fano $4$-fold with $\delta_X\leq 1$. If there exists a $2$-dimensional face $\sigma$ of $\Eff(X)$ such that $\sigma\cap\Mov(X)\neq\{0\}$, then $\rho_X\leq 11$.
\end{lemma}
\begin{proof}
We keep the same notation as in  Lemma \ref{useful} and its proof. We have $d\leq 2$, thus $\rho_X\leq 12$ by  Lemma \ref{useful}.  Suppose by contradiction that $\rho_X=12$. Then $d=2$ and $\dim\N(F,\w{X})=\rho_F=10$, 
therefore $F\cong S\times\pr^1$, where $S$ is a del Pezzo surface with $\rho_S=9$  \cite[\S 12.6]{fanoEMS}. 
Hence $S$ and $F$ are covered by a family of (non rational) curves of anticanonical degree $1$.
This means that there is a covering
family of  curves in $\w{X}$ whose general member is an irreducible curve,
of anticanonical degree $1$, and contracted by $\tilde{f}$. As $\tilde{f}$ is $K$-negative, 
we deduce that
every curve of the family has anticanonical degree $1$. Therefore $\w{X}$ is covered by integral curves of anticanonical degree $1$, and Lemma \ref{SQM}(3) implies that $X\dasharrow \w{X}$ is an isomorphism.

Thus $F\subset X$ is a prime divisor with $\codim\N(F,X)=2$, but this contradicts the assumption $\delta_X\leq 1$; this concludes the proof.
\end{proof}
\begin{corollary}\label{noquasiel}
Let $X$ be a smooth Fano $4$-fold with $\rho_X\geq 12$ and $\delta_X\leq 1$, and let $f\colon X\dasharrow Y$ be an elementary divisorial rational contraction.
Then every one-dimensional face of $\Eff(Y)$ is generated by a fixed prime divisor.
\end{corollary}
\begin{proof}
We have to show that no one-dimensional face of $\Eff(Y)$ is contained in $\Mov(Y)$. By Lemma \ref{bijection}(1)-(2), this is equivalent to showing that there is no $2$-dimensional face $\sigma$ of $\Eff(X)$ such that $[\Exc(f)]\in\sigma$ and $\sigma\cap\Mov(X)\neq \{0\}$, and this follows from Lemma \ref{macau}.
\end{proof}
\section{Blowing-up many points}\label{secmanypoints}
\begin{parg}\label{introsecmanypoints} In this section we show the following result, which implies Th.~\ref{manypoints}.
\begin{thm}\label{manypointsgeneral}
 Let $X$ be a smooth Fano $4$-fold. Suppose that there exists 
a normal and $\Q$-factorial projective variety $Y$ with $\rho_Y=1$, and 
 a composition $f\colon \w{X}\to Y$ of blow-ups of smooth points, such
that 
 $X$ and $\w{X}$ are isomorphic in codimension one. 

Then 
$\rho_X\leq 9$, and if $\rho_X\geq 8$, then $Y\cong\pr^4$.
\end{thm}
Let us outline our strategy.
We first show that 
$h^0(X,-K_X)=h^0(Y,-K_Y)-15r$, where $r$ is the number of blow-ups in $\w{X}\to Y$, hence $r=\rho_X-1$ (Prop.~\ref{drop}). This is based on the Hirzebruch-Riemann-Roch formula and on the explicit description of SQM's of $X$ in Th.~\ref{SQM0}. Since it is known that  $h^0(Y,-K_Y)\leq 126$ (Th.~\ref{rho1}), this yields $r\leq 8$ and hence $\rho_X\leq 9$. To give a finer bound on $\rho_X$ in the case $Y\not\cong\pr^4$, we use a refined bound on $h^0(Y,-K_Y)$ and an analysis of curves of low anticanonical degree in $Y$ (Lemmas \ref{locus} and \ref{3}).
\end{parg}
\subsection{Behaviour of $\chi(X,-K_X)$ under blow-up in a $4$-fold}\label{chi}
\noindent The goal of this section if to prove the following.
\begin{proposition}\label{drop}
Let $X$ be a smooth Fano $4$-fold, $X\dasharrow\w{X}$ a SQM, and
$\w{X}\to Y$ a composition of $r$ blow-ups of a smooth point, 
with $Y$ Fano. Then we have:
$$h^0(X,-K_X)=h^0(Y,-K_Y)-15r.$$
\end{proposition}
\noindent In order to prove the above relation, we need to analyse the behaviour of the Euler characteristic of the anticanonical sheaf of a smooth $4$-fold under the blow-up of a smooth point and under a $K$-negative flip.

We recall that if $X$ is a smooth projective $4$-fold and $D$ is a divisor on $X$, the Hirzebruch-Riemann-Roch formula \cite[Cor.~15.2.1]{fultonint} gives 
\stepcounter{thm}
\begin{equation}\label{RR0}
 \chi(X,D)=\frac{1}{24}\Big(D^4-2K_X\cdot D^3+D^2\left(K_X^2+c_2(X)\right)-D\cdot K_X\cdot c_2(X)\Big)+\chi(X,\ma{O}_X).
\end{equation}
In particular for the anticanonical divisor we have: 
\stepcounter{thm}\begin{equation}\label{RR}
\chi(X,-K_{X})=\frac{1}{12}\bigl(2(-K_{X})^4+
(-K_{X})^2\cdot c_2(X)\bigr)+\chi(X,\ma{O}_{X}).
\end{equation}
\begin{remark}[\cite{fultonint}, Ex.~15.4.3]\label{c2}
Let $Y$ be a smooth projective $4$-fold and  $f\colon X\to Y$ the blow-up along a smooth irreducible subvariety $A\subset Y$ of codimension $d$. Let $E$ be the exceptional divisor, so that $E=\pr_A(\ma{N}^*_{A/Y})\subset X$ where $\ma{N}_{A/Y}$ is the normal bundle\footnote{We use Grothendieck's notation for the projectivization of a vector bundle.} of $A$ in $Y$. Moreover let $j\colon E\hookrightarrow X$ be the inclusion, and set $\zeta:=c_1(\ma{O}_{\ma{N}^*_{A/Y}}(1))\in H^2(E,\Z)$. Then we have $-K_X=f^*(-K_Y)-(d-1)E$ and
$$c_2(X)=f^*c_2(Y)-j_*\left((d-1)f_{|E}^*(c_1(A))+\frac{d(d-3)}{2}\zeta+(d-2)
f_{|E}^*c_1(\ma{N}_{A/Y})\right).$$
\end{remark}
\begin{lemma}\label{blowup}
Let $Y$ be a smooth projective $4$-fold
 and $f\colon X\to Y$ the blow-up at a point. Let $E\subset X$ be the exceptional divisor (so that $E\cong\pr^3$), $H\subset E$ a plane, and  $[H]\in H^4(X,\Z)$ the fundamental class of $H$.
Then we have:
\begin{align*}
 c_2(X)&=f^*c_2(Y)-2[H],& (-K_X)^4&=(-K_Y)^4-81,\\
 (-K_X)^2\cdot c_2(X)&=(-K_Y)^2\cdot c_2(Y)-18,
 &
\chi(X,-K_X)&=\chi(Y,-K_Y)-15.
\end{align*}
\end{lemma}
\begin{proof}
The formula for  $c_2(X)$ follows from Rem.~\ref{c2}. The intersections are given by a simple computation, using also \cite[Rem.~3.2.4]{fultonint} in $E$,
and the last equality follows from  \eqref{RR}.
\end{proof}
The following Lemmas \ref{blowupcurve} and \ref{blowupexcplane} are proven in the same way as Lemma \ref{blowup}.
\begin{lemma}\label{blowupcurve}
Let $Y$ be a smooth projective $4$-fold
 and  $f\colon X\to Y$ the blow-up  along a smooth irreducible curve $C$ of genus $g(C)$. Let $E\subset X$ be the exceptional divisor,  $F\subset E$ a fiber of the $\pr^2$-bundle $E\to C$, and $[F]\in H^4(X,\Z)$ the fundamental class of $F$. Set $d(C):=-K_Y\cdot C+2-2g(C)$.
Then we have:
\begin{align*}
c_2(X)&=f^*c_2(Y)-d(C)[F], &(-K_X)^2\cdot c_2(X)&=(-K_Y)^2\cdot c_2(Y)-4d(C),\\
(-K_X)^4&=(-K_Y)^4-16d(C), &\chi(X,-K_X)&=\chi(Y,-K_Y)-3d(C).
\end{align*}
\end{lemma}
\begin{lemma}\label{blowupexcplane}
Let $Y$ be a smooth projective $4$-fold
 and  $f\colon X\to Y$ the blow-up  along an exceptional plane.
Let $E\subset X$ be the exceptional divisor, so that
 $E\cong \pr^1\times \pr^2$. Set $S:=\{pt\}\times\pr^2\subset E$ and $T:=\pr^1\times l\subset E$, where $l\subset\pr^2$ is a line, and let $[S],[T]\in H^4(X,\Z)$ be the fundamental classes of $S$ and $T$ respectively. Then we have:
\begin{align*}
c_2(X)&=f^*c_2(Y)+[S]-2[T],&
(-K_X)^4&=(-K_Y)^4-17,\\
 (-K_X)^2\cdot c_2(X)&=(-K_Y)^2\cdot c_2(Y)-2,&
\chi(X,-K_X)&=\chi(Y,-K_Y)-3.
\end{align*}
\end{lemma}
\begin{corollary}\label{flip}
Let $X$ be a smooth Fano $4$-fold, $\ph\colon X\dasharrow \w{X}$ a SQM, and 
 $s$ the number of connected components of $X\smallsetminus\dom\ph$. Then $\chi(X,-K_X)=\chi(\w{X},-K_{\w{X}})$ and $(-K_X)^4=(-K_{\w{X}})^4+s$.
\end{corollary}
\begin{proof}
Consider the factorization $X\stackrel{f}{\leftarrow}\widehat{X}\stackrel{g}{\to}\w{X}$ of $\ph$ given by Th.~\ref{SQM0}.
By Lemmas \ref{blowupcurve} and \ref{blowupexcplane} we have
$\chi(\widehat{X},-K_{\widehat{X}})=\chi(X,-K_X)-3s=\chi(\w{X},-K_{\w{X}})-3s$
and $(-K_{\widehat{X}})^4=(-K_X)^4-17s=(-K_{\w{X}})^4-16s$,
which yields the statement.
\end{proof}
\begin{proof}[Proof of Prop.~\ref{drop}]Notice that $\w{X}$ is smooth by Th.~\ref{SQM0}, and hence $Y$ is smooth.
By Cor.~\ref{flip} and Lemma \ref{blowup} we get
$$h^0(X,-K_X)=\chi(X,-K_X)=\chi(\w{X},-K_{\w{X}})=\chi(Y,-K_Y)-15r=h^0(Y,-K_Y)-15r.$$
\end{proof}
\subsection{Proof of Theorems \ref{manypoints} and \ref{manypointsgeneral}}
\noindent We start by analysing some properties of curves of low anticanonical degree in $Y$.
\begin{lemma}\label{locus}
Let $X$ be a smooth Fano $4$-fold, $\ph\colon X\dasharrow\w{X}$ a SQM, and $f\colon\w{X}\to Y$ a composition of $r$ blow-ups of a smooth point.
Let $C\subset Y$ be an integral curve.
\begin{enumerate}[(1)]
\item If $-K_Y\cdot C=1$, then $C$ is contained in the open subset where the map $f\circ\ph\colon X\dasharrow Y$ is an isomorphism. 
\item If $-K_Y\cdot C=2$ and $C\cap f(\Exc(f))\neq\emptyset$, then $C=f(\ell)$, $\ell\subset\w{X}$ an exceptional line.
\item
If $-K_Y\cdot C=3$, then $C\cap f(\Exc(f))=\emptyset$.
\end{enumerate}
\end{lemma}
\begin{proof}Notice that $\w{X}$ is smooth by Th.~\ref{SQM0}, and hence $Y$ is smooth.
Using the formula for the canonical divisor of a blow-up, we see that
$$f^*(-K_Y)=-K_{\w{X}}+3E,$$
where $E$ is an effective divisor with $\Supp(E)=\Exc(f)$.
 Notice that $C\not\subset f(\Exc(f))$, because $\dim f(\Exc(f))=0$;
let $\w{C}\subset\w{X}$ be the transform of $C$. We have
$$-K_Y\cdot C=-K_{\w{X}}\cdot\w{C}+3E\cdot\w{C},$$
and either 
 $C\cap f(\Exc(f))=\emptyset$, or $E\cdot\w{C}>0$.

Suppose that 
$1\leq -K_Y\cdot C\leq 3$ and  $C\cap f(\Exc(f))\neq\emptyset$.
Then 
$-K_{\w{X}}\cdot\w{C}\leq 0$, and Lemma \ref{SQM}(2) implies that $-K_{\w{X}}\cdot\w{C}=-1$ and that $\w{C}$ is an exceptional line. Moreover
$$-K_Y\cdot C=-1+3E\cdot\w{C}\equiv 2\mod 3,$$
thus $-K_Y\cdot C=2$. This gives (2) and (3).
If $-K_Y\cdot C=1$, we have shown that 
$C\cap f(\Exc(f))=\emptyset$ and $-K_{\w{X}}\cdot\w{C}=1$. Thus Lemma \ref{SQM}(3) yields that $\w{C}\subset\dom(\ph^{-1})$, and this gives (1).
\end{proof}
\begin{lemma}\label{3}
Let $X$ be a smooth Fano $4$-fold, $\ph\colon X\dasharrow\w{X}$ a SQM, and $f\colon\w{X}\to Y$ a composition of $r$ blow-ups of a smooth point, with $Y$ Fano.
Let $C\subset Y$ be a connected curve with $-K_Y\cdot C=3$. Then $C\cap f(\Exc(f))=\emptyset$.
\end{lemma}
\begin{proof}
If $C$ is integral, this is Lemma \ref{locus}(3). 
If $C$ is not integral, 
suppose by contradiction that $C\cap f(\Exc(f))\neq \emptyset$, and let 
 $C_1$ be an irreducible component of $C$ intersecting $f(\Exc(f))$. Since $Y$ is Fano, we have $-K_Y\cdot C_1\in\{1,2\}$, and Lemma \ref{locus}(1)-(2) yields that $-K_Y\cdot C_1=2$ and the transform $\w{C}_1$ of $C_1$ in $\w{X}$ is an exceptional line.

Hence we must have
$C=C_1\cup C_2$ with $C_2$ integral and $-K_Y\cdot C_2=1$. By Lemma \ref{locus}(1) the curve $C_2$ is contained in the open subset where the map $f\circ\ph\colon X\dasharrow Y$ is an isomorphism, but this is impossible, 
because $C_1\cap C_2\neq\emptyset$ and $C_1$ is the image of an exceptional line in $\w{X}$.
\end{proof}
\begin{lemma}\label{sequence}
 Let $X$ be a smooth Fano $4$-fold, $\ph\colon X\dasharrow\w{X}$ a SQM, and  $f\colon\w{X}\to Y$ a composition of $r$ blow-ups of a smooth point.
Then $f$ is the blow-up of $r$ distinct points in $Y$.
\end{lemma}
\begin{proof}
We proceed by induction on $r\geq 0$. If $r=0$ or $r=1$, there is nothing to prove. For the general case, consider a factorization of $f$ as in the following diagram:
$$\xymatrix{{\w{X}}\ar[r]_g\ar@/^1pc/[rr]^f&Z\ar[r]_h&Y
}$$
where $h$ is a single blow-up and $g$ is a composition of $r-1$ blow-ups. 

By the induction assumption, $g$ is the blow-up of $r-1$ distinct points $p_1,\dotsc,p_{r-1}\in Z$. 
We have $\Exc(h)\cong\pr^3$, and if $C\subset \Exc(h)$ is a line, then $-K_Z\cdot C=3$. By Lemma \ref{locus}(3), we have $\{p_1,\dotsc,p_{r-1}\}\cap C=\emptyset$, therefore $p_i\not\in\Exc(h)$ for every $i=1,\dotsc,r-1$. This yields the statement.
\end{proof}
The next result follows from the classification of Fano $3$-folds in the case of index $>1$, and has been proved by Hwang \cite{hwang4folds} in the case of index $1$.
\begin{thm}\label{rho1}
Let $Y$ be a smooth Fano $4$-fold with $\rho_Y=1$. We have the following: 
\begin{enumerate}[(1)]
\item $h^0(Y,-K_Y)\leq 126$, with equality if and only if $Y\cong\pr^4$;
\item if $Y\not\cong\pr^4$, then either $h^0(Y,-K_Y)\leq 105$, or $Y$ has index $1$ and the minimal anticanonical degree of a
 covering family of rational curves in $Y$ is $3$.
\end{enumerate}
\end{thm}
\begin{proof}
If $Y\cong\pr^4$, then $h^0(Y,-K_Y)=126$. If $Y$ is isomorphic to a quadric in $\pr^5$, then $h^0(Y,-K_Y)=105$. If $Y$ is neither $\pr^4$ nor a quadric, then by Kobayashi-Ochiai's theorem, the index of $Y$ is at most $3$.

If $Y$ has index $3$, let $H\in\Pic(Y)$ be an ample generator of $\Pic(Y)$. By Kodaira vanishing we have $\chi(Y,-H)=0$, which using Hirzebruch-Riemann-Roch \eqref{RR0} yields 
$$H^2\cdot c_2(Y)=2H^4+12$$
and again by Hirzebruch-Riemann-Roch \eqref{RR} we get
$$h^0(Y,-K_Y)=15H^4+10.$$
It has been proven by Fujita that the linear system $|H|$ contains a smooth element $Z$ \cite[Th.~2.3.2 and references therein]{fanoEMS}.
 Then $Z$ is a Fano $3$-fold with index $2$, $\rho_Z=1$, and $(-K_Z)^3=8H^4$. It is known by classification that $(-K_Z)^3\leq 40$ \cite[\S 12.2]{fanoEMS}, thus $h^0(Y,-K_Y)\leq 85$.

If $Y$ has index $2$, we proceed in the same way as in the case of index $3$: if $H$ is an ample generator of 
$\Pic(Y)$, we have
$$H^2\cdot c_2(Y)=H^4+24\quad\text{ and }\quad h^0(Y,-K_Y)=3H^4+9.$$
It has been proven by Wilson that the linear system $|H|$ contains a smooth element $Z$ \cite[Th.~2.3.15 and references therein]{fanoEMS}.
Then $Z$ is a Fano $3$-fold with index $1$, $\rho_Z=1$, and $H^4=(-K_Z)^3$. By classification $(-K_Z)^3\leq 22$, thus $h^0(Y,-K_Y)\leq 75$.

Finally, if $Y$ has index $1$, by \cite[Th.~3]{hwang4folds} we have $h^0(Y,-K_Y)\leq 121$, and
$h^0(Y,-K_Y)\leq 97$ unless the minimal anticanonical degree of a
covering family of rational curves of $Y$ is $3$. This yields the statement.
\end{proof}
\begin{proof}[Proof of Th.~\ref{manypointsgeneral}]Notice that $\w{X}$ is smooth by Th.~\ref{SQM0}, and hence $Y$ is smooth.
If $Y\cong\pr^4$, then $h^0(Y,-K_Y)=126$ and $0\leq h^0(X,-K_X)=126-15r$ by Prop.~\ref{drop}, which yields $r\leq 8$ and $\rho_X\leq 9$.

Suppose that $Y\not\cong\pr^4$ and that $\rho_X>1$. Then $f$ is not an isomorphism, thus $f(\Exc(f))\neq\emptyset$. 
Hence Lemma \ref{3} implies that $Y$ cannot have a covering family of  curves of anticanonical degree $3$, and Th.~\ref{rho1} yields $h^0(Y,-K_Y)\leq 105$.

It has been shown by Kawamata \cite[Th.~5.2]{kawasian} that $h^0(X,-K_X)>0$ (in fact $h^0(X,-K_X)\geq 2$, see \cite[Th.~1.2]{floris}).
Therefore Prop.~\ref{drop} yields 
$$0< h^0(X,-K_X)=h^0(Y,-K_Y)-15r\leq 105-15r=15(7-r),$$
thus $r\leq 6$ and $\rho_X\leq 7$. This concludes the proof.
\end{proof}
\section{Bounding $\rho_X$ when every fixed divisor is of type $(3,0)^{sm}$}\label{every}
\begin{parg}[Fixed prime divisors and MMP's]\label{pargfixed}
Let $X$ be a projective, normal, and $\Q$-factorial Mori dream space. 
If $D\subset X$ is a fixed prime divisor, running a MMP for $D$ yields a 
diagram
$$X\stackrel{\ph}{\dasharrow}\w{X}\stackrel{f}{\la} Y$$
where $\ph$ is a SQM that factors as a sequence of $D$-negative flips,  
 and $f$ is an elementary divisorial contraction
 with exceptional locus the transform of $D$. We will always denote a MMP for a fixed prime divisor in this way: as a SQM followed by an elementary divisorial contraction.
\end{parg}
\begin{definition}[fixed divisors of type $(3,0)^{sm}$]\label{30}
Let $X$ be a smooth Fano $4$-fold and $D\subset X$ a fixed prime divisor. 
We say that $D$ is {\bf of type $\mathbf{(3,0)^{sm}}$} is there exists a 
SQM $X\dasharrow X'$ such that $X'$ is the blow-up of a projective $4$-fold $Y'$ at a smooth point,
 with exceptional divisor
 the transform of $D$.
\end{definition}
This notion is crucial for the rest of the paper; notice that in our main result, Theorem \ref{onepoint}, the assumption on $X$ can be reformulated by saying that $X$ has a fixed prime divisor of type $(3,0)^{sm}$. 

The goal of section \ref{every} is to prove the following intermediate result.
\begin{thm}\label{everyfixeddiv}
Let $X$ be a smooth Fano $4$-fold such that every fixed prime divisor is of type $(3,0)^{sm}$.
Then $\rho_X\leq 11$, and if $\rho_X\geq 10$, then $X$ has an elementary rational contraction of fiber type.
\end{thm}
We first study properties of fixed prime divisors of type $(3,0)^{sm}$, in section \ref{every1}. Then we study Fano $4$-folds where every fixed prime divisor is of  type $(3,0)^{sm}$, in section \ref{every2}. We conclude section \ref{every2} with the proof of Th.~\ref{everyfixeddiv}.
\subsection{Fixed divisors of type $(3,0)^{sm}$}\label{every1}
\noindent In this section we study properties of fixed divisors of type $(3,0)^{sm}$. We show that given a fixed divisor $D$ of type $(3,0)^{sm}$ in $X$, there exists a SQM $X\dasharrow \w{X}$ such that $\w{X}$ is the blow-up of another smooth \emph{Fano} $4$-fold $Y$ at a point (Prop.~\ref{fanotarget}). We use this to show that whenever we have $X\stackrel{\psi}{\dasharrow}X'\stackrel{g}{\to}Y'$, where $\psi$ is a SQM and $g$ is an elementary divisorial contraction with exceptional locus the transform of $D$, then $g$ is of type $(3,0)^{sm}$ (Prop.~\ref{tofix}).  We conclude the section by relating 
 fixed prime divisors of $X$ to fixed prime divisors of $Y$ (Lemma \ref{transform}).
\begin{lemma}\label{firstproperties}
Let $X$ be a smooth Fano $4$-fold and $D\subset X$ a fixed prime divisor of type $(3,0)^{sm}$. Let  $\psi\colon X\dasharrow X'$ be a SQM and $g\colon X'\to Y'$ the blow-up of a smooth point $p\in Y'$, with $\Exc(g)$ the transform of $D$.
\begin{enumerate}[(1)]
\item
For every exceptional plane $L\subset X\smallsetminus \dom\psi$
we have either  $D\cap L=\emptyset$, or $L\subset D$ and $D\cdot C_L<0$, where $C_L$ is a line in $L\cong\pr^2$.
Moreover $X\smallsetminus \dom\psi$ contains all exceptional planes contained in $D$.
\item For every irreducible curve $C\subset Y'$ such that $p\in C$ we have $-K_{Y'}\cdot C\geq 2$ and $-K_{Y'}\cdot C\neq 3$. If $-K_{Y'}\cdot C=2$, then $C=g(\ell)$, $\ell\subset X'$ an exceptional line with $\Exc(g)\cdot\ell=1$.
\item The point $p$ can belong to (at most) finitely many integral curves of anticanonical degree $2$, and cannot belong to an exceptional plane.
\end{enumerate}
\end{lemma}
\begin{proof}
Let  $L\subset X\smallsetminus \dom\psi$ be an exceptional plane, and $\ell\subset X'\smallsetminus \dom\psi^{-1}$ the corresponding exceptional line
(see Th.~\ref{SQM0}).
Notice that every curve contained in $\Exc(g)$ is contracted by $g$ and has positive anticanonical degree; in particular $\Exc(g)$ cannot contain $\ell$.  Hence either $\Exc(g)\cdot\ell>0$, or $\Exc(g)\cap \ell=\emptyset$; respectively in $X$ we have either $D\cdot C_L<0$ and $L\subset D$, or $D\cap L=\emptyset$.

Consider now an exceptional plane $L_0\subset D$. If $L_0$ intersects $\dom\psi$, then by Rem.~\ref{SQM2} it intersects the exceptional planes in $X\smallsetminus\dom\psi$ in (at most) finitely many points, therefore a general line $C_{L_0}\subset L_0\cong\pr^2$ is contained in $\dom\psi$. If $C_{L_0}'\subset\Exc(g)$ is the transform of $C_{L_0}$, then $-K_{X'}\cdot C_{L_0}'=-K_X\cdot C_{L_0}=1$, which is impossible because every curve in $\Exc(g)$ has anticanonical degree at least $3$. Thus we have (1).

Finally, (2) follows from Lemma \ref{locus} (and its proof), and
(3) follows from (2).
\end{proof}
\begin{proposition}\label{fanotarget}
Let $X$ be a smooth Fano $4$-fold and $D\subset X$ a fixed prime divisor of type $(3,0)^{sm}$. 
Then there exist a SQM $\ph\colon X\dasharrow \w{X}$ and  $f\colon \w{X}\to Y$ a blow-up of a smooth point $p\in Y$, with $\Exc(f)$ the transform of $D$, such that:
\begin{enumerate}[(1)]
\item $X\smallsetminus\dom\ph$ is the union of all exceptional planes contained in $D$;
\item $X\stackrel{\ph}{\dasharrow}\w{X}\stackrel{f}{\to}Y$ is a MMP for $D$;
\item $Y$ is Fano.
\end{enumerate}
Moreover $\ph\colon X\dasharrow \w{X}$ and $f\colon \w{X}\to Y$ are uniquely 
determined by one of the properties (1), (2), or (3) above.
\end{proposition}
\begin{proof}
Let $L_1,\dotsc,L_r$ be the exceptional planes contained in $D$.
By Def.~\ref{30}, there exist a SQM $\psi\colon X\dasharrow X'$ and  a blow-up of a smooth point $g\colon X'\to Y'$ with $\Exc(g)$ the transform of $D$.
By  Th.~\ref{SQM0} and Lemma \ref{firstproperties}(1) we have 
$$X\smallsetminus\dom\psi=L_1\cup\cdots\cup L_r\cup L_1'\cup\cdots\cup L_s',$$
where $L_i'$ is an exceptional plane disjoint from $D$ for $i=1,\dotsc,s$, and $s\geq 0$.
\begin{parg} We construct $\ph\colon X\dasharrow\w{X}$ and compare $\ph$ and $\psi$.

Let us run a MMP for $D$ in $X$, by choosing at each step a $D$-negative \emph{small} extremal ray of the cone of effective curves, and we stop when there are no more $D$-negative small extremal rays. This yields a SQM $\ph\colon X\dasharrow \w{X}$ such that $X\smallsetminus\dom\ph$ is a union of exceptional planes contained in $D$, thus
$$X\smallsetminus\dom\ph\subseteq L_1\cup\cdots\cup L_r.$$
Up to reordering, we can assume that $X\smallsetminus\dom\ph= L_1\cup\cdots\cup L_t$, with $t\in\{0,\dotsc,r\}$.
Then $L_{t+1},\dotsc,L_r,L_1',\dotsc,L_s'\subset\dom\ph$, hence their images 
$$\ph(L_{t+1}),\dotsc,\ph(L_r),\ph(L_1'),\dotsc,\ph(L_s')$$ are  pairwise disjoint exceptional planes in $\w{X}$.
Moreover, if  $\w{D}\subset \w{X}$ is the transform of $D$, then 
$\ph(L_{t+1}),\dotsc,\ph(L_r)$ are contained in $\w{D}$, while $\ph(L_1'),\dotsc,\ph(L_s')$ are disjoint from $\w{D}$.

 Let us consider the diagram:
$$\xymatrix{X \ar@{-->}[r]_{\ph}\ar@/^1pc/@{-->}[rr]^{\psi} &{\w{X}}\ar@{-->}[r]_{\xi}&{ X'}\ar[r]^g&Y'}$$
where $\xi:=\psi\circ\ph^{-1}$ is a SQM, and by construction
$$\w{X}\smallsetminus\dom\xi=\ph(L_{t+1})\cup\cdots\cup
\ph(L_r)\cup\ph(L_1')\cup\cdots\cup\ph(L_s').$$
Notice that $ X'\smallsetminus\dom(\xi^{-1})$ is a finite union of exceptional lines.

As $\Exc(g)\subset X'$ cannot contain exceptional lines, it intersects $X'\smallsetminus\dom(\xi^{-1})$ in at most finitely many points.
Let $T\subset\Exc(g)\cong\pr^3$ be a general plane and $\Gamma\subset T$ a line.
 Then $T\subset\dom\xi^{-1}$, so if $\w{\Gamma}\subset \w{D}\subset\w{X}$ is the transform of $\Gamma$, we have 
\stepcounter{thm}
\begin{equation}\label{bonassola}
\w{D}\cdot\w{\Gamma}=\Exc(g)\cdot\Gamma=-1\quad\text{and}\quad -K_{\w{X}}\cdot\w{\Gamma}=-K_{ X'}\cdot\Gamma=3.\end{equation}
\end{parg}
\begin{parg}
We show that $\w{D}$ cannot contain exceptional planes. 

By contradiction, 
let $L\subset \w{D}$ be an exceptional plane, and let $C_L\subset L$ be a line.
 We have $-K_{\w{X}}\cdot C_L=1$, and using \eqref{bonassola} it is easy to check that the classes $[C_L]$ and $[\w{\Gamma}]$ in $\N(\w{X})$ cannot be proportional. Since both curves $C_L$ and $\w{\Gamma}$ are contained in $\w{D}$, we get $\dim\N(\w{D},\w{X})>1$.

Being $D$ a fixed divisor, $\w{D}$ cannot be nef, so there is at least one $\w{D}$-negative extremal ray $R$ of $\NE(\w{X})$, which by construction is divisorial. The extremal ray $R$ is $K$-negative, by Lemma \ref{SQM}(1).

Since $\dim\N(\w{D},\w{X})>1$, the extremal ray $R$ cannot be of type $(3,0)$ (see Rem.~\ref{basic}(1)). 
Recall that $\w{\Gamma}$ is the transform of a general line in $\Exc(g)\cong\pr^3$, thus
$[\w{\Gamma}]\not\in R$, otherwise $R$ would be of type $(3,0)$. Therefore
  the contraction of $R$ must be finite on the tranform $\w{T}\subset\w{D}$ of the plane $T\subset\Exc(g)$.  This implies that $R$ is of type $(3,2)$ and that $\dim\N(\w{D},\w{X})=2$ (see Rem.~\ref{basic}(3)).

Thus $\w{D}$ is covered by integral curves of anticanonical degree $1$ (see \ref{basic0}), and
by Lemma \ref{SQM}(3) we have $\w{D}\subset\dom\ph^{-1}$ and $D\subset\dom\ph$.
On the other hand $\ph$ is a composition of $D$-negative flips, 
thus $\ph$ must be the identity and $X=\w{X}$. Moreover we have shown that every $D$-negative extremal ray of $\NE(X)$
is divisorial of type $(3,2)$. By Lemma \ref{torino}(1), $-K_X+D$ is nef.

Notice that since $\dim\N(D,X)=2$, there are at most two $D$-negative extremal rays. On the other hand $(-K_X+D)\cdot\w{\Gamma}= 2$ by \eqref{bonassola}, thus $(-K_X+D)_{|D}$ is not numerically trivial. Again by Lemma \ref{torino}(1), this implies that $R$ is the unique $D$-negative extremal ray of $\NE(X)$, and 
$R=(-K_X+D)^{\perp}\cap\NE(X)$.

Let us consider now the exceptional plane $L\subset D$. The contraction of $R$ is finite on $L$ by Rem.~\ref{basic}(4), thus $[C_L]\not\in R$ and $(-K_X+D)\cdot C_L>0$. Since $-K_X\cdot C_L=1$, this implies that $D\cdot C_L\geq 0$, contradicting
 Lemma \ref{firstproperties}(1).

We conclude that $\w{D}$ cannot contain exceptional planes.
\end{parg} 
\begin{parg} 
We are now ready to show the statement.
Since $\w{D}$  does not contain exceptional planes, we must have $t=r$,
$$X\smallsetminus\dom\ph=L_1\cup\cdots\cup L_r,\ \text{ and } \ \w{X}\smallsetminus\dom\xi=\ph(L_1')\cup\cdots\cup\ph(L_s'),$$
thus we have (1).
Recall that the exceptional planes
$\ph(L_1'),\dotsc,\ph(L_s')$ are
 disjoint from $\w{D}$, so that  $\w{D}\subset\dom\xi$. In particular $\w{D}\cong\Exc(g)\cong\pr^3$; this implies that the $\w{D}$-negative extremal ray of $\NE(\w{X})$ is unique, and yields a contraction $f\colon \w{X}\to Y$ which is locally isomorphic to $g$, and hence is the blow-up of $Y$ at a smooth point $p$. Thus we have also (2).

We are left to show that $Y$ is Fano. If $C\subset Y$ is an irreducible curve containing $p$, then $-K_{Y}\cdot C>0$ by Lemma \ref{firstproperties}(2). If instead $p\not\in C$, then the transform $\w{C}\subset \w{X}$ of $C$ is disjoint from $\w{D}$, and hence it is not an exceptional line, thus $-K_{Y}\cdot C=-K_{\w{X}}\cdot \w{C}>0$. Notice that $Y$ is a Mori dream space, hence $\NE(Y)$ is closed; we conclude that $-K_Y$ is ample, and we have (3).
 
Let us consider now the map  $\xi_Y:=g\circ\xi\circ f^{-1}\colon Y\dasharrow Y'$ as in the diagram:
$$\xymatrix{X \ar@{-->}[r]_{\ph}\ar@/^1pc/@{-->}[rr]^{\psi} &{\w{X}}\ar[d]_{f}\ar@{-->}[r]_{\xi}&{ X'}\ar[d]^g\\
&Y\ar@{-->}[r]^{\xi_Y}&{Y'}}
$$
Then $\xi_Y$ is birational, and it is an isomorphism in codimension one, namely it is a SQM of the smooth Fano $4$-fold $Y$. Moreover $$Y\smallsetminus\dom(\xi_Y)=f(\ph(L_1'))\cup\cdots\cup f(\ph(L_s')).$$
For the last statement of the Theorem, let us notice that the following statements are equivalent:
\begin{enumerate}[$(i)$]
\item $X\smallsetminus\dom\psi=L_1\cup\cdots\cup L_r$;
\item $\psi$ factors as a sequence of $D$-negative flips;
\item $Y'$ is Fano;
\item $s=0$;
\item $\xi$ and $\xi_Y$ are isomorphisms.
\end{enumerate}
\end{parg}
\vspace{-20pt}
\end{proof} 
\begin{definition}[the curve $C_D$]\label{prima}
Let $X$ be a smooth Fano $4$-fold and $D\subset X$ a fixed divisor of type $(3,0)^{sm}$. Consider the MMP for $D$:
$$X\stackrel{\ph}\dasharrow\w{X}\stackrel{f}{\la}Y$$
as in Prop.~\ref{fanotarget}.
We denote by $C_D\subset D$ the transform of a general line in $\Exc(f)\cong\pr^3$. 

Notice that $D\cdot C_D=-1$ and $-K_X\cdot C_D=3$.
\end{definition}
\begin{proposition}
\label{tofix}
Let $X$ be a smooth Fano $4$-fold and $D\subset X$ a fixed divisor of type $(3,0)^{sm}$.
Let $\psi\colon X\dasharrow X'$ be a SQM and $g\colon X'\to Y'$ an elementary divisorial contraction with exceptional divisor the transform of $D$.
Then $g$ is of type $(3,0)^{sm}$,
$C_D\subset\dom\psi$, and $\psi(C_D)$ is a line in $\Exc(g)\cong\pr^3$.
\end{proposition}
\begin{proof}
Let $X\stackrel{\ph}{\dasharrow}\w{X}\stackrel{f}{\to}Y$ be the MMP for $D$ as in
  Prop.~\ref{fanotarget}. We have a commutative diagram:
$$\xymatrix{X \ar@{-->}[r]_{\ph}\ar@/^1pc/@{-->}[rr]^{\psi} &{\w{X}}\ar[d]_{f}\ar@{-->}[r]_{\xi}&{X'}\ar[d]^g\\
&Y\ar@{-->}[r]^{\xi_Y}&{Y'}}
$$
where the birational maps $\xi:=\psi\circ\ph^{-1}\colon \w{X}\dasharrow X'$ and 
 $\xi_Y:=g\circ\xi\circ f^{-1}\colon Y\dasharrow Y'$ are isomorphisms in codimension $1$. Set $\w{D}:=\Exc(f)\subset\w{X}$, $T:=\w{X}\smallsetminus\dom\xi$, and $Z:=Y\smallsetminus\dom\xi_Y$.

Since $Y$ is a smooth Fano $4$-fold,  by Th.~\ref{SQM0} the variety $Y'$ is smooth and $Z$ is a union of exceptional planes. The point $p\in Y$ cannot be contained in an exceptional plane by Lemma \ref{firstproperties}(3), thus $p\in\dom\xi_Y$ and $\w{D}\cap f^{-1}(Z)=\emptyset$.

Since $\xi$ is an isomorphism in codimension $1$, this implies that there is a (non-empty) open subset of $\Exc(g)$ contracted to the point $\xi_Y(p)\in Y'$, and hence that $g$
 is of type $(3,0)$.

By the commutativity of the diagram above, we see that $\xi$ is regular on $f^{-1}(\dom\xi_Y)\smallsetminus\w{D}$, hence 
$$T=\w{X}\smallsetminus\dom\xi\subseteq \bigl(\w{X}
\smallsetminus f^{-1}(\dom\xi_Y)\bigr)\cup\w{D}=f^{-1}(Y\smallsetminus
\dom\xi_Y)\cup\w{D}=  f^{-1}(Z)\cup\w{D}.$$ 
If $T\cap\w{D}$ were non-empty, then it would be a connected component of $T$, entirely contained in $\w{D}$; it is not difficult to see that this is impossible, by considering a factorization of $\xi$ as a sequence of flips, and using that $f$ is of type $(3,0)$.
Thus $\w{D}\subset\dom\xi$, and $\xi$ is an isomorphism between $\Exc(f)$ and $\Exc(g)$. This shows the statement.
\end{proof}
The property above is quite special and depends on the special type of divisorial contraction that we are considering, that is, the blow-up of a smooth point. In general, the same fixed divisor in $X$ may become exceptional for different types of elementary divisorial contractions, as in the following example.
\begin{example}\label{special}
Set $Y:=\pr_{\pr^2}(\ma{O}\oplus\ma{O}(1)\oplus\ma{O}(2))$, and let $S\subset Y$ be the section of $Y\to\pr^2$ with normal bundle $\ma{N}_{S/Y}=\ma{O}_{\pr^2}(-1)\oplus\ma{O}_{\pr^2}(-2)$. Let $f\colon X\to Y$ be the blow-up of $Y$ along $S$,
 and set $D:=\Exc(f)\subset X$, so that $D\cong\Bl_{pt}\pr^3$.

Then $X$ is a toric Fano $4$-fold with $\rho_X=3$ (this is $D_3$ in Batyrev's list \cite{bat2}), and $D\subset X$ is a fixed prime divisor which is the exceptional divisor of the elementary divisorial contraction $f\colon X\to Y$, of type $(3,2)^{sm}$. 

On the other hand there is a $D$-negative flip $X\dasharrow X'$ such that if $D'\subset X'$ is the transform of $D$, then  $D'\cong\pr^3$,  $\ma{N}_{D'/X'}\cong\ma{O}_{\pr^3}(-2)$, and there is an elementary divisorial contraction $g\colon X'\to Y'$, of type $(3,0)$, with $\Exc(g)=D'$.

Let us finally notice that $\rho_X=3$ in this example; we will see in section \ref{hard} that such a behaviour cannot happen when $\rho_X\geq 6$.
\end{example}
\begin{lemma}\label{hongkong}
Let $X$ be a smooth Fano $4$-fold and $D_1,D_2\subset X$ two distinct fixed prime divisors of type $(3,0)^{sm}$. If $D_1+D_2$ is not movable, then $D_1\cdot C_{D_2}=D_2\cdot C_{D_1}=0$.
\end{lemma}
\begin{proof}
Set $G:=D_1+D_2$. We run a MMP for $G$; being $G$ effective, the MMP must end with $G$ becoming nef. On the other hand, $G$ is not movable, thus some divisorial contraction must occur in the MMP; let us consider the first one. We have
$$X\stackrel{\ph}{\dasharrow} X'\stackrel{f}{\la}Y,$$
where $\ph$ is a SQM and $f$ is a $G'$-negative elementary divisorial contraction, with $G'=D_1'+D_2'$, $D_i'\subset X'$ the transform of $D_i\subset X$.
Then one of the two intersections $D_i'\cdot \NE(f)$ must be negative; up to exchanging the divisors $D_i$, we can assume that $D_1'\cdot\NE(f)<0$, hence $D_1'=\Exc(f)$. 

By Prop.~\ref{tofix} $f$ is of type $(3,0)^{sm}$,  $C_{D_1}\subset\dom\ph$, and if $C'_{D_1}:=\ph(C_{D_1})$, then $C'_{D_1}$ is a line in $D'_1\cong\pr^3$; in particular $D_2\cdot C_{D_1}=D_2'\cdot C_{D_1}'$.
We get
$$0>G'\cdot C_{D_1}'=(D_1'+D_2')\cdot C_{D_1}'=D_2'\cdot C_{D_1}'-1,$$
which yields $D_2'\cdot C_{D_1}'\leq 0$. On the other hand $D_2'\neq D_1'$, thus 
 $D_2'\cdot C_{D_1}'\geq 0$, and we conclude that  $D_2\cdot C_{D_1}=D_2'\cdot C_{D_1}'=0$. 
Since $D_1'\cong\pr^3$,
this implies that the two divisors $D_1'$ and $D_2'$ are disjoint in $X'$. In $X$, their transforms $D_1$ and $D_2$ may meet only along exceptional planes in $X\smallsetminus\dom(\ph)$. Since a general $C_{D_2}\subset D_2$ is disjoint from any exceptional plane, we also have $D_1\cdot C_{D_2}=0$.
\end{proof}
\begin{lemma}\label{transform}
Let $X$ be a smooth Fano $4$-fold, $D\subset X$ a fixed prime divisor of type $(3,0)^{sm}$, and 
$X\stackrel{\ph}\dasharrow\w{X}\stackrel{f}{\la}Y$ the
 MMP for $D$ as in Prop.~\ref{fanotarget}.

Let $B_Y\subset Y$ be a fixed prime divisor, and $B_X\subset X$ its transform. Then $B_X$ is a fixed divisor, and $B_X$ is of type $(3,0)^{sm}$ if and only if $B_Y$ is of type $(3,0)^{sm}$.
\end{lemma}
\begin{proof}
By Lemma \ref{bijection} we know that  $B_X$ is a fixed divisor in $X$ such that   $\langle [D],[B_X]\rangle\cap\Mov(X)=\{0\}$; in particular $D+B_X$ is not movable.

Since $B_Y$ is fixed, there are a SQM $\psi_Y\colon Y\dasharrow\widehat{Y}$ and an elementary divisorial contraction $g_Y\colon\widehat{Y}\to W$ such that $\Exc(g_Y)$ is the transform of $B_Y$ (see \ref{pargfixed}).

The rational contraction $\psi_Y\circ f\colon \w{X}\dasharrow\widehat{Y}$ can be factored as $\w{X}\stackrel{\psi_X}{\dasharrow}\widehat{X}\stackrel{\hat{f}}{\to}\widehat{Y}$, where $\psi_X\colon\w{X}\dasharrow\widehat{X}$ is a SQM and $\hat{f}\colon \widehat{X}\to\widehat{Y}$ is an elementary birational contraction. 
$$
\xymatrix{X\ar@{-->}[r]^{\ph}&{\w{X}}\ar@{-->}[r]^{\psi_X}\ar[d]_f&{\widehat{X}}
\ar[d]^{\hat{f}}\\
&Y\ar@{-->}[r]^{\psi_Y}&{\widehat{Y}}\ar[r]^{g_Y}&W}
$$
Since $D$ is contracted in $Y$, $\hat{f}$ must be divisorial with exceptional divisor the transform of $D$. By assumption $D$ is of type $(3,0)^{sm}$,  hence Prop.~\ref{tofix} yields that $\hat{f}$ is again the blow-up of a smooth point $p\in\widehat{Y}$; moreover $C_D\subset\dom(\psi_X\circ\ph)$ and $[(\psi_X\circ\ph)(C_D)]\in\NE(\hat{f})$. 
\begin{parg}\label{p}
Let $B_{\widehat{X}}\subset \widehat{X}$ be the transform of $B_X$. We show that if either $B_X$ or $B_Y$ is of type $(3,0)^{sm}$, then $$B_{\widehat{X}}\cap\Exc(\hat{f})=\emptyset.$$

Indeed suppose that $B_X$ is of type $(3,0)^{sm}$. Then by Lemma \ref{hongkong} we have $B_X\cdot C_D=0$ in $X$ and hence $B_{\widehat{X}}\cdot\NE(\hat{f})=0$ in $\widehat{X}$. This implies that  $B_{\widehat{X}}\cap\Exc(\hat{f})=\emptyset$, because $\Exc(\hat{f})\cong\pr^3$.

Suppose instead that $B_Y$ is of type $(3,0)^{sm}$. Then  $g_Y$ is of type $(3,0)^{sm}$ by Prop.~\ref{tofix}; thus we can apply Lemma \ref{sequence} to $g_Y\circ\hat{f}\colon \widehat{X}\to W$, and deduce that
$p\not\in \Exc(g_Y)$, $B_{\widehat{X}}=\hat{f}^{-1}(\Exc(g_Y))$ and $B_{\widehat{X}}\cap\Exc(\hat{f})=\emptyset$. 
\end{parg}
\begin{parg}
Suppose now that either $B_X$ or $B_Y$ is of type $(3,0)^{sm}$, and
let us consider $h:=g_Y\circ\hat{f}\colon\widehat{X}\to W$. The cone $\NE(h)$ has dimension $2$, and one of its extremal rays is $\NE(\hat{f})$. 
Let $g_X\colon \widehat{X}\to Z$ be the contraction of the second extremal ray of
$\NE(h)$ and $k\colon Z\to W$ the induced contraction; notice that both $g_X$ and $k$ are birational.
$$
\xymatrix{X\ar@{-->}[r]^{\ph}&{\w{X}}\ar@{-->}[r]^{\psi_X}\ar[d]_f&{\widehat{X}}
\ar[d]_{\hat{f}}\ar[dr]^h\ar[r]^{g_X}&Z\ar[d]^k\\
&Y\ar@{-->}[r]^{\psi_Y}&{\widehat{Y}}\ar[r]^{g_Y}&W}
$$
We have  $B_{\widehat{X}}\cap\Exc(\hat{f})=\emptyset$ by \ref{p},  hence $\Exc(h)=B_{\widehat{X}}\cup\Exc(\hat{f})$; moreover it is not difficult to see that 
$B_{\widehat{X}}\cdot\NE(g_X)<0$, so that $\Exc(g_X)\subseteq B_{\widehat{X}}$. Notice that 
$g_X(\Exc(\hat{f}))\cong\pr^3$ is a Cartier divisor in $Z$, with negative normal bundle; this implies that $\Exc(k)=g_X(\Exc(\hat{f}))$ and hence $g_X$ must be divisorial with $\Exc(g_X)=B_{\widehat{X}}$.
Therefore $g_X$ and $g_Y$ are locally isomorphic,
and $g_X$ is of type $(3,0)^{sm}$ if and only if $g_Y$ is; this yields the statement.
\end{parg}
\vspace{-20pt}
\end{proof}
\subsection{Fano $4$-folds where every fixed divisor is of type $(3,0)^{sm}$}\label{every2}
\begin{parg}\label{introevery2}
\noindent In this section we study Fano $4$-folds $X$ where every fixed prime divisor is of type $(3,0)^{sm}$, and we prove Th.~\ref{everyfixeddiv} stating that $\rho_X\leq 11$. 

Let us outline the strategy. We know that if $X$ has an elementary rational contraction of fiber type, then $\rho_X\leq 11$ by Th.~\ref{elemft}. One can always find a sequence of flips and elementary divisorial contractions $X\dasharrow Y$ such that $Y$ has an elementary rational contraction of fiber type $Y\dasharrow Z$. Using the fact that every fixed prime divisor of $X$ is of type $(3,0)^{sm}$, we show that if $\dim Z\geq 2$, then we can ``lift'' this elementary rational contraction of fiber type to $X$, and hence apply   Th.~\ref{elemft} (Lemma \ref{lifting}). Otherwise, we reduce to the case where $\rho_Y\leq 2$, and we show that $X$ is isomorphic in codimension one to $\Bl_{p_1,\dotsc,p_r}Y$  (Lemma \ref{rho1o2}). Finally if $\rho_Y=1$ we conclude by Th.~\ref{manypoints}; if $\rho_Y=2$ we follow the same proof as for Th.~\ref{manypoints}, and for this we need to analyse $Y$ in order to bound $h^0(Y,-K_Y)$ (Lemma \ref{caserho2}).
\end{parg}
\begin{lemma}\label{factor}
Let $X$ be a projective, normal, and $\Q$-factorial Mori dream space, and $f\colon X\dasharrow Y$  a sequence of flips and elementary divisorial contractions (in arbitrary order).
Set $r:=\rho_X-\rho_Y$. Then $f$ can be factored as
$$X\dasharrow X_0\la X_1\la\cdots\la X_{r-1}\la X_r=Y,$$
where $X\dasharrow X_0$ is a SQM, and $X_{i-1}\to X_i$ is an elementary divisorial contraction for every $i=1,\dotsc,r$. 
\end{lemma}
\begin{proof}
We proceed by induction on $r$. If $r=0$, then $f$ is a SQM and there is nothing to prove.
Assume that $r>0$, and let us factor $f$ as in the following diagram:
$$\xymatrix{{X}\ar@{-->}[r]\ar@{-->}@/^1pc/[rr]^f&Z\ar@{-->}[r]_g&{Y,}
}
$$
where both maps are sequences of flips and elementary divisorial contractions, and $\rho_Z-\rho_Y=1$. We can factor $g$ as follows:
$$\xymatrix{{Z}\ar@{-->}[r]\ar@{-->}@/^1pc/[rr]^g&{\w{Z}}\ar[r]_{\tilde{g}}&{Y,}
}$$
where $Z\dasharrow\w{Z}$ is a SQM, and $\tilde{g}$ is regular. Then $\tilde{g}$ is an elementary birational contraction, and cannot be small because $Y$ is $\Q$-factorial; thus it is divisorial. 

Let us now consider the composite map $X\dasharrow \w{Z}$. By applying the induction assumption, we get the statement.
\end{proof}
\begin{lemma}\label{several}
Let $X$ be a smooth Fano $4$-fold such that every fixed prime divisor is of type $(3,0)^{sm}$.
Let $X\dasharrow Y$ be  a sequence of flips and elementary divisorial contractions (in arbitrary order).  We have the following:
\begin{enumerate}[(1)]
\item
$Y$ is smooth, it is a SQM of a smooth Fano $4$-fold $Y'$, and every fixed prime divisor in $Y'$ is of type $(3,0)^{sm}$;
\item
 every
divisorial contraction occurring in the sequence is of type $(3,0)^{sm}$;
\item set $r:=\rho_X-\rho_Y$; there are $r$ distinct points $p_1,\dotsc,p_r\in Y$ such that $X$ and $\Bl_{p_1,\dotsc,p_r}Y$ are isomorphic in codimension one.
\end{enumerate}
\end{lemma}
\begin{proof}
To show (1) and (2), we proceed by induction on $r=\rho_X-\rho_Y$. If $r=0$, then $X\dasharrow Y$ is a SQM, and $Y$ is smooth by Th.~\ref{SQM0}.
Let us consider the general case. We factor $X\dasharrow Y$ as
$$X\dasharrow X'\stackrel{f}{\la} Z\stackrel{\ph}{\dasharrow}Y,$$
where $X\dasharrow X'$ is a SQM,
$f\colon X'\to Z$ is  an elementary divisorial contraction, and $\ph\colon Z\dasharrow Y$ is a sequence of flips and elementary divisorial contractions. 

Let $D\subset X$ be the transform of $\Exc(f)$. Then $D$ is a fixed prime divisor, and by assumption it is of type $(3,0)^{sm}$. Therefore Prop.~\ref{tofix} and \ref{fanotarget} yield that $Z$ is smooth, $f$ is of type $(3,0)^{sm}$, and 
there exists a SQM $\psi\colon Z'\dasharrow Z$ where $Z'$ is a  
 Fano $4$-fold. Moreover, by Lemma \ref{transform}, every fixed prime divisor in $Z'$ is of type $(3,0)^{sm}$. By the induction assumption applied to the map $\ph\circ\psi\colon Z'\dasharrow Y$, we get (1) and (2).

We show (3). By Lemma \ref{factor} the map $X\dasharrow Y$ can be factored as 
$$X\dasharrow\w{X}\stackrel{g}{\la} Y,$$
where $X\dasharrow\w{X}$ is a SQM and  $g$ is a composition of $r$ elementary divisorial contractions. By (2), every such contraction is of type $(3,0)^{sm}$, thus the statement follows from Lemma \ref{sequence}.
\end{proof}
The following is a ``version with flips'' of the well-known 2-ray game.
\begin{lemma}\label{lifting0}
Let $X$ be a projective, normal, and $\Q$-factorial Mori dream space, $f\colon X\to Y$ an elementary divisorial contraction, and $g\colon Y\to Z$ an elementary contraction of fiber type. Set $h:=g\circ f\colon X\to Z$, $D:=\Exc(f)\subset X$, and assume that $\codim h(D)\geq 2$. 

Then there exists a SQM $\ph\colon X\dasharrow\w{X}$ such that $\tilde{h}:=h\circ\ph^{-1}\colon\w{X}\to Z$ is regular and factors 
as $\w{X}\stackrel{\tilde{g}}{\to}T\stackrel{\tilde{f}}{\to}Z$, where $\tilde{g}$ is an elementary contraction of fiber type, and $\tilde{f}$ is an elementary divisorial contraction with $\Exc(\tilde{f})=\tilde{g}(\w{D})$,
$\w{D}\subset\w{X}$ the transform of $D$.
$$\xymatrix{{X}\ar[d]_f\ar[dr]^h\ar@{-->}^{\ph}[rr]&&{\w{X}}\ar[dl]_{\tilde{h}}
\ar[d]^{\tilde{g}}\\    
{Y}\ar[r]^g&Z&{T}\ar[l]_{\quad\tilde{f}}
}$$
\end{lemma}
\begin{proof}
Set $W:=h(D)$. Since $W$ has codimension at least $2$ and $g$ is elementary of fiber type, the inverse image $g^{-1}(W)\subset Y$ has codimension at least $2$ (see for instance \cite[Rem.~2.9]{eff}), so 
 $D$ is the unique divisorial component of $h^{-1}(W)$.

Now we run a MMP for $-D$ on $X$, relative to $h$.
Since we are considering $D$-positive extremal rays, contained in $\NE(h)$, the exceptional loci of all the maps in the MMP are contained in $h^{-1}(W)$. In particular, the MMP cannot end with a $D$-positive fiber type contraction, so it must end with $-D$ becoming $h$-nef. 

On the other hand, since we are considering $D$-positive extremal rays, $D$ can never be the exceptional divisor of a divisorial contraction in the MMP. Since there are no other divisorial components in $h^{-1}(W)$, there can be no divisorial contraction in the MMP. This means that the MMP yields a SQM $\ph\colon X\dasharrow\w{X}$ such that the composition $\tilde{h}:=h\circ\ph^{-1}\colon\w{X}\to Z$ is regular, and if
 $\w{D}\subset\w{X}$ is the transform of $D$, then $-\w{D}$ is $\tilde{h}$-nef.

Notice that $\NE(\tilde{h})$ is a $2$-dimensional cone.
Since $\tilde{h}$ is of fiber type, $-\w{D}$ cannot be $\tilde{h}$-ample. On the other hand $\tilde{h}(\w{D})=W$ is not a divisor, thus $\w{D}$ is not the pull-back of a divisor under $\tilde{h}$, hence $\w{D}$ is not $\tilde{h}$-trivial (again see \cite[Rem.~2.9]{eff}). 
 Therefore $\w{D}^{\perp}\cap\NE(\tilde{h})$ is an extremal ray of $\NE(\w{X})$, which yields an elementary contraction $\tilde{g}\colon \w{X}\to T$ such that $\tilde{h}$ factors through $\tilde{g}$. Let $\tilde{f}\colon T\to Z$ be the induced elementary contraction.

Since $\w{D}$ does not dominate $Z$, if $C\subset \w{X}$ is a curve in a general fiber of $\tilde{h}$, then $C\cap\w{D}=\emptyset$, hence $\w{D}\cdot C=0$ and $[C]\in\NE(\tilde{g})$. This shows that $\tilde{g}$ is of fiber type.
Moreover $D_T:=\tilde{g}(\w{D})$ is a prime divisor such that $\w{D}=\tilde{g}^*(D_T)$ and $\tilde{f}(D_T)=W$. It is not difficult to see that $D_T\cdot \NE(\tilde{f})<0$, in particular $\tilde{f}$ is birational. On the other hand $Z$ is $\Q$-factorial, hence $\tilde{f}$ cannot be small. We conclude that $\tilde{f}$ is divisorial with $\Exc(\tilde{f})=D_T$.
\end{proof}
\begin{lemma}\label{lifting}
Let $X$ be a smooth Fano $4$-fold such that every fixed prime divisor is of type $(3,0)^{sm}$.
 Let $\alpha\colon X\dasharrow Y$ be a sequence of flips and elementary divisorial contractions (in arbitrary order), and $\beta\colon Y\dasharrow Z$ an 
elementary rational contraction of fiber type.
If  $\dim Z\in\{2,3\}$, then $\beta\circ\alpha$ can be factored as in the 
 following diagram:
$$\xymatrix{
X\ar@{-->}[r]^{\alpha}\ar@{-->}[d]_{\tilde{\beta}}&Y\ar@{-->}[d]^{\beta}\\
T\ar[r]^{\tilde{\alpha}}&Z
}$$
where $\tilde{\beta}$ is an elementary rational contraction of fiber type, and $\tilde{\alpha}$ is a sequence of elementary divisorial contractions.
\end{lemma}
\begin{proof}
We proceed by induction on $\rho_X-\rho_Y$. If $\rho_X=\rho_Y$, then $\alpha$ is a SQM, and we can take $T=Z$, $\tilde{\beta}=\beta\circ\alpha$, and $\tilde{\alpha}=\text{Id}_Z$.
Suppose now that  $\rho_X-\rho_Y>0$; we can factor $\alpha$ as
$$
\xymatrix{{X}\ar@{-->}[r]_{f}\ar@{-->}@/^1pc/[rr]^{\alpha}&
{W}\ar@{-->}[r]_{\psi}&{Y}}$$
where both $f$ and $\psi$ are sequences of flips and elementary divisorial contractions, and $\rho_X-\rho_{W}=1$.

By Lemma \ref{several}, $W$ is smooth, there is a SQM $W\dasharrow \w{W}$ such that $\w{W}$ is a Fano $4$-fold, and every fixed divisor of $\w{W}$ is of type $(3,0)^{sm}$. Up to composing $f$ and $\psi$ with a SQM, we can assume that $W=\w{W}$.

Then we can apply the induction assumption to $\psi\colon W\dasharrow Y$ and $\beta\colon Y\dasharrow Z$, and deduce that
there exists a commutative diagram:
$$\xymatrix{
W\ar@{-->}[r]^{\psi}\ar@{-->}[d]_g&Y\ar@{-->}[d]^{\beta}\\
{Z'}\ar[r]^{\chi}&Z
}$$
where $g$ is
an elementary rational contraction of fiber type, and $\chi$ is a sequence of elementary divisorial contractions. 

Let $\ph\colon{W}\dasharrow W'$ be a SQM such that the composition $g':=g\circ\ph^{-1}\colon W'\to Z'$ is regular, and consider the rational contraction $\ph\circ f\colon X\dasharrow W'$.
Finally
 let $\xi\colon {X}\dasharrow X'$ be a SQM such that $f':=\ph\circ f\circ\xi^{-1}\colon X'\to W'$ is regular.
$$\xymatrix{
X\ar@{-->}[d]_f\ar@{-->}[r]^{\xi}\ar@{-->}[rd] & {X'}\ar[d]^{f'}\ar[rd]^h\ar@{-->}[rr]^{\eta}&&{\w{X}}\ar[dl]_{\tilde{h}}\ar[d]^{\tilde{g}}\\
W\ar@/_1pc/@{-->}[rr]_{g}\ar@{-->}[r]^{\ph} & {W'}\ar[r]^{g'}&{Z'}&{T}\ar[l]^{\tilde{f}}
}$$
Set $h:=g'\circ f'\colon X'\to Z'$. By Lemma \ref{several} $f'$ is of type $(3,0)^{sm}$; set $D:=\Exc(f')$. Then $h(D)$ is a point in $Z'$ and $\dim Z'=\dim Z\geq 2$, so by Lemma \ref{lifting0} there exists a SQM $\eta\colon X'
\dasharrow\w{X}$ such that the composition $\tilde{h}:=h\circ\eta^{-1}\colon\w{X}\to Z'$ is regular and factors as $\tilde{f}\circ\tilde{g}$, where $\tilde{g}\colon \w{X}\to T$
is an elementary contraction of fiber type, and $\tilde{f}\colon T\to Z'$ is an elementary divisorial contraction.
Now we set $\tilde{\beta}:=\tilde{g}\circ\eta\circ\xi$ and $\tilde{\alpha}:=\chi\circ\tilde{f}$, and we have the statement.
\end{proof}
\begin{lemma}\label{rho1o2}
Let $X$ be a smooth Fano $4$-fold such that every fixed prime divisor is of type $(3,0)^{sm}$.
Then one of the following holds:
\begin{enumerate}[$(i)$]
\item 
$X$ has an elementary rational contraction of fiber type;
\item  
there is a SQM $X\dasharrow \w{X}$ such that $\w{X}=\Bl_{p_1,\dotsc,p_r}Y$, where $Y$ is smooth and $p_1,\dotsc,p_r\in Y$ are distinct points, and one of the following holds:
\begin{enumerate}
\item[$(iia)$] $\rho_Y=1$;
\item[$(iib)$] $\rho_Y=2$ and $Y$ has two distinct elementary rational contractions $Y\dasharrow\pr^1$.
\end{enumerate}
\end{enumerate}
\end{lemma}
\begin{proof}
Let $A$ be an ample divisor on $X$. By running a MMP for $-A$, we get a diagram
$$X\dasharrow Y\stackrel{f}{\la} Z$$
where $X\dasharrow Y$ is a sequence of flips and elementary divisorial contractions, and $f$ is an elementary contraction of fiber type.
By Lemma \ref{several}(1)-(3), $Y$ is smooth and there is a 
SQM $X\dasharrow \w{X}$ such that $\w{X}=\Bl_{p_1,\dotsc,p_r}Y$, where 
$p_1,\dotsc,p_r\in Y$ are distinct points.

If $\dim Z>1$, then we have $(i)$ by Lemma \ref{lifting}.
If $Z=\{pt\}$, then $\rho_Y=1$ and we have $(iia)$.
Suppose that $\dim Z=1$, so that $Z\cong\pr^1$ and $\rho_Y=2$. 
Then $f^*(\Nef(Z))$ is a one-dimensional face of $\Mov(Y)$; let $\sigma$ be the other one-dimensional face of $\Mov(Y)$, and let $g\colon Y\dasharrow Z'$ be the associated elementary rational contraction (namely $\sigma=g^*(\Nef(Z'))$), so that $\rho_{Z'}=1$.

Since $g^*(\Nef(Z'))$ is in the boundary of $\Mov(Y)$, either $g$ is divisorial, or it is of fiber type. If it is divisorial, then we replace $X\dasharrow Y$ by the composition $X\dasharrow Z'$, and we have again $(iia)$. If $g$ is of fiber type
and $\dim Z'>1$, we have again $(i)$ as before. Otherwise $Z'\cong\pr^1$, and we have $(iib)$.
\end{proof}
\begin{lemma}\label{caserho2}
Let $X$ be a smooth Fano $4$-fold such that every fixed prime divisor is of type $(3,0)^{sm}$, and $\rho_X\geq 3$.
If we are in case $(iib)$ of Lemma \ref{rho1o2}, then $h^0(Y,-K_Y)\leq 40$.
\end{lemma}
\begin{proof}
By Lemma \ref{several}(1), up to composing $X\dasharrow Y$ with a SQM of $Y$, we can assume that $Y$ is Fano. 
For both rational contractions $Y\dasharrow\pr^1$, let us consider a SQM $Y_i\dasharrow Y$ such that the composition $f_i\colon Y_i\to\pr^1$ is regular and $K$-negative (see Rem.~\ref{Kneg}).
We have a diagram:
$$\xymatrix{{Y_1}\ar[d]_{f_1}&Y\ar@{-->}[l]\ar@{-->}[r]&{Y_2}\ar[d]^{f_2}\\
{\pr^1}&&{\pr^1}
}$$
Fix $i\in\{1,2\}$ and let $F_i\subset Y_i$ be a general fiber of $f_i$. Since $f_i$ is $K$-negative, $F_i$ is a smooth Fano $3$-fold. 
 Let $\w{F}_i\subset Y$ be the transform of $F_i$. 
\begin{parg}\label{uff}
We show that $F_i$ cannot have a covering family of rational curves of anticanonical degree $3$. This implies that $F_i$ contains a curve $\Gamma_i$ with $c_i:=-K_{F_i}\cdot \Gamma_i\in\{1,2,4\}$ \cite[Cor.~IV.1.15]{kollar}; notice that $F_i\cdot\Gamma_i=0$ and $-K_{Y_i}\cdot\Gamma_i=c_i$.

By contradiction, if $F_i$ has a covering family of rational curves of anticanonical degree $3$, then $Y_i$ has a covering family $V$ of rational curves, of anticanonical degree $3$, and such that $f_i(C_v)$ is a point of every $v\in V$. 

Consider now the map $X\dasharrow Y_i$. By Lemma \ref{several}(3), we know that there are $r:=\rho_X-2$ distinct points $p_1,\dotsc,p_r\in Y_i$ such that 
$X$ and $\Bl_{p_1,\dotsc,p_r}Y_i$ are isomorphic in codimension one. Notice that $r>0$, because $\rho_X\geq 3$.
Then, by considering a curve $C_{v}$ of the family containing $p_1$,
 we reach a contradiction as in the proof of Lemma \ref{3} (note that even if 
$Y_i$ is not Fano, every irreducible component of $C_v$ has positive anticanonical degree, because
$f_i$ is
$K$-negative).
\end{parg}
\begin{parg}
We show that $Y\dasharrow Y_i$ cannot be an isomorphism. Indeed if for instance $Y=Y_1$ and $f_1\colon Y\to\pr^1$, let $R$ be the extremal ray of $\NE(Y)$ different from $\NE(f_1)$. 
Then $R$ is either of type $(4,1)$ (if $Y=Y_2$), or small (if $Y\dasharrow Y_2$ is not an isomorphism), and in this last case it is of type $(2,0)$ (see Lemma \ref{lorenzo}). Thus in any case the contraction of $R$ has a fiber $Z$ with $\dim Z\geq 2$, but $f_1$ must be finite on $Z$, which gives a contradiction.

We conclude that $\NE(Y)$ is generated by two small extremal rays. In particular $Y$ contains curves of anticanonical degree $1$ (see Lemma \ref{lorenzo}), hence it has index $1$.
\end{parg}
\begin{parg}
We have $\Eff(Y)=\langle[\w{F}_1],[\w{F}_2]\rangle$, so there exist positive integers $n,a_1,a_2$ such that
$$n(-K_Y)=a_1\w{F}_1+a_2\w{F}_2,$$
and we can assume that $\gcd(n,a_1,a_2)=1$. 

Since $Y_i$ is rationally connected, by \cite{GraberHarrisStarr} there exists $C_i\subset Y_i$ section of $f_i$, so that $F_i\cdot C_i=1$.
Now let $\widehat{F}_2\subset Y_1$ be the transform of $F_2$. In $Y_1$ we have 
\stepcounter{thm}
\begin{equation}\label{Y1}
n(-K_{Y_1})=a_1F_1+a_2\widehat{F}_2,
\end{equation}
 and intersecting with $C_1$ we get 
$$n(-K_{Y_1}\cdot C_1)=a_1+a_2\widehat{F}_2\cdot C_1.$$
This implies that $\gcd(a_2,n)$ divides $a_1$, and hence that $\gcd(a_2,n)=1$. Similarly we get that $\gcd(a_1,n)=1$.

Set $d:=\gcd(a_1,a_2)$ and  $a_i:=da_i'$ for $i=1,2$. 
We have
$$n(-K_Y)=d(a'_1\w{F}_1+a'_2\w{F}_2)$$
and $\gcd(n,d)=1$. This implies that $d=1$, because $Y$ has index $1$. 
Now intersecting with $\Gamma_1$ in $Y_1$ (see \ref{uff}) we get from \eqref{Y1}:
$$nc_1=a_2\widehat{F}_2\cdot\Gamma_1,$$
which implies that $a_2$ divides $c_1$ and hence that $a_2\in\{1,2,4\}$.
Similarly we deduce that $a_1\in\{1,2,4\}$.
\end{parg}
\begin{parg}
Since $\gcd(a_1,a_2)=1$, 
we can assume that $a_1=1$, so that \eqref{Y1} becomes $n(-K_{Y_1})=F_1+a_2\widehat{F}_2$ and hence
$$-K_{Y_1}-F_1=a_2\widehat{F}_2-(n-1)(-K_{Y_1}).$$
We know that $h^0(Y_1,-K_{Y_1})=h^0(Y,-K_Y)>0$ by \cite[Th.~5.2]{kawasian};
thus we get
\begin{gather*}
h^0(Y_1,-K_{Y_1}-F_1)=h^0(Y_1,a_2\widehat{F}_2-(n-1)(-K_{Y_1}))\leq 
h^0(Y_1,a_2\widehat{F}_2)
=h^0(Y_2,a_2F_2)\\=h^0(Y_2,f_2^*\ma{O}_{\pr^1}(a_2))
=h^0(\pr^1,\ma{O}_{\pr^1}(a_2))=a_2+1\leq 5.\end{gather*}
Now the exact sequence
$$0\la\ma{O}_{Y_1}(-K_{Y_1}-F_1)\la\ma{O}_{Y_1}(-K_{Y_1})\la\ma{O}_{F_1}(-K_{F_1})\la 0$$
yields 
$$h^0(Y,-K_Y)=h^0(Y_1,-K_{Y_1})\leq h^0(Y_1,-K_{Y_1}-F_1)+h^0(F_1,-K_{F_1}).$$
On the other hand $F_1$ is a smooth Fano $3$-fold; we have $(-K_{F_1})^3\leq 64$ \cite[Ch.~12]{fanoEMS}, and by Hirzebruch-Riemann-Roch \cite[Ex.~15.2.5]{fultonint}
$$h^0(F_1,-K_{F_1})=\chi(F_1,-K_{F_1})=\frac{1}{2}(-K_{F_1})^3+3\leq 35.$$
Therefore $h^0(Y,-K_Y)\leq 40$, and this concludes the proof.
\end{parg}
\vspace{-20pt}
\end{proof}
\begin{proof}[Proof of Th.~\ref{everyfixeddiv}]
We apply Lemma \ref{rho1o2}. If $(i)$ holds, then $X$ has a rational contraction of fiber type, and $\rho_X\leq 11$  by Th.~\ref{elemft}.
If $(iia)$ holds, then $\rho_X\leq 9$ by Th.~\ref{manypoints}.
If $(iib)$ holds, then either $\rho_X=2$, or $\rho_X\geq 3$ and $h^0(Y,-K_Y)\leq 40$ by Lemma \ref{caserho2}. Hence by Prop.~\ref{drop} we have 
$$0\leq h^0(X,-K_X)=h^0(Y,-K_Y)-15r\leq 40-15r$$
thus $r\leq 2$ and $\rho_X\leq 4$.
\end{proof}
\section{Fixed divisors in Fano $4$-folds with $\rho\geq 6$}\label{hard}
\subsection{The type of a fixed prime divisor}\label{tf}
\noindent Fixed prime divisors in Fano $4$-folds with $\rho\geq 6$ are described in \cite{eff}; there are four possible types, one of which -- type $(3,0)^{sm}$ -- has already been defined and studied in section \ref{every}. In this section we recall this result from \cite{eff}, and we 
define the type of a fixed prime divisor.
\begin{thm}[\cite{eff}, Th.~3.15 and Rem.~3.16]\label{typefixed}
Let $X$ be a smooth Fano $4$-fold with $\rho_X\geq 6$, 
 $D\subset X$ a fixed prime divisor, and
consider a MMP for $D$:
$$X\stackrel{\ph}{\dasharrow}\widetilde{X}\stackrel{f}{\la}Y.$$
Then one of the following holds:
\begin{enumerate}[$(i)$]
\item  $\ph$ is an isomorphism and $f$ is of type $(3,2)$;
\item $Y$ is  Fano and $f$ is of type $(3,1)^{sm}$;
\item $Y$ is Fano and $f$ is of type $(3,0)^{sm}$;
\item $Y$ is (singular) Fano and $f$ is of type $(3,0)^Q$.
\end{enumerate}
Moreover, if $f$ is not of type $(3,2)$, then $\ph$ factors as a sequence of at least $\rho_X-4$ $K$-negative flips.
\end{thm}
\begin{corollary}\label{uniqueness}
Let $X$ be a smooth Fano $4$-fold with $\rho_X\geq 6$, and $D\subset X$ a fixed prime divisor. Then one of the following holds:
\begin{enumerate}[$(i)$]
\item there is a unique MMP $X\stackrel{\ph}{\dasharrow}\w{X}\stackrel{f}{\to}Y$ for $D$, up the order of the flips in $\ph$; in particular both the maps $\ph\colon X\dasharrow\w{X}$ and $f\colon \w{X}\to Y$ in Th.~\ref{typefixed} are uniquely determined by $D$;
\item  $\rho_X=6$, $\delta_X=3$, 
$\dim\N(D,X)=3$, and $\NE(X)$ has $n\geq 2$ $D$-negative extremal rays $R_1,\dotsc,R_n$, all of type $(3,2)$. The possible MMP's for $D$ are $f_i\colon X\to Y_i$, where $f_i$ is the contraction of $R_i$.  
\end{enumerate}
\end{corollary}
\noindent Whenever we are in case $(i)$ of Cor.~\ref{uniqueness}, we can talk about ``the MMP for $D$''. Notice that we are always in case $(i)$ if we assume that $\rho_X\geq 7$ or $\rho_X=6$ and $\delta_X\leq 2$, and also if $D$ is not the locus of an extremal ray of $\NE(X)$ of type $(3,2)$.
\begin{proof}
Let  $X\stackrel{\ph}{\dasharrow}\w{X}\stackrel{f}{\to}Y$ 
and $X\stackrel{\ph_2}{\dasharrow}\w{X}_2\stackrel{f_2}{\to}Y_2$ 
be two MMP's for $D$.
Then the induced birational map $\psi\colon Y\dasharrow Y_2$ is a SQM.

If both $Y$ and $Y_2$ are Fano, then $\psi$ is an isomorphism (because the pull-back of the ample divisor $-K_{Y_2}$ stays ample on $Y$). Thus we may assume that $Y=Y_2$ and  $\psi=\Id_Y$, namely that $f\circ\ph=f_2\circ\ph_2\colon X\dasharrow Y$. 
On the other hand, since the rational contraction 
$g:=f\circ\ph\colon X\dasharrow Y$ is elementary and divisorial, the SQM $\xi\colon X\dasharrow X'$ such that $g\circ\xi^{-1}\colon X'\to Y$ is regular is unique (up to isomorphism of the target). Therefore we conclude that 
there is an isomorphism $\chi\colon\w{X}\to\w{X}_2$ such that $\ph_2=\chi\circ\ph$ and $f_2=f\circ\chi^{-1}$, so we are in case $(i)$.

Suppose now that $Y$ and $Y_2$ are not both Fano; we can assume for simplicity that $Y$ is not Fano. Then, by Th.~\ref{typefixed}, $\ph$ is an isomorphism and $f$ is of type $(3,2)$.
 
Let $R_1,\dotsc,R_n$ be the $D$-negative extremal rays of $\NE(X)$. Each $R_i$ is of type $(3,2)$, by Rem.~\ref{pisa}(2); moreover since $Y$ is not Fano, we must have $n\geq 2$ by \cite[Prop.~3.4(i)]{wisn}. This yields $\rho_X=6$, $\delta_X=3$, and $\dim\N(D,X)=3$, by Rem.~\ref{pisa}(3).
Since there are no $D$-negative small extremal rays, $\ph_2$ must be an isomorphism too, and $f_2$ is the contraction of one of the extremal rays $R_1,\dotsc,R_n$. This yields 
$(ii)$.
\end{proof}
\begin{definition}[the type of a fixed prime divisor in a Fano $4$-fold with $\rho\geq 6$]\label{deftypefixed}
Let $X$ be a smooth Fano $4$-fold with $\rho_X\geq 6$, and
 $D\subset X$ a fixed prime divisor.

We say that $D$ is {\bf of type $\mathbf{(3,2)}$} if it is the locus of an extremal ray of $\NE(X)$ of type $(3,2)$.
Otherwise, let
 $X\stackrel{\ph}{\dasharrow}\w{X}\stackrel{f}{\to}Y$ be the MMP for $D$.
We say that $D$ is:
\begin{enumerate}[--]
\item {\bf of type $\mathbf{(3,1)^{sm}}$} if we are in case $(ii)$ of Th.~\ref{typefixed};
\item {\bf  of type $\mathbf{(3,0)^{sm}}$} if we are in case $(iii)$ of Th.~\ref{typefixed};
\item {\bf of type $\mathbf{(3,0)^Q}$} if we are in case $(iv)$ of Th.~\ref{typefixed}.
\end{enumerate}
\end{definition}
Cor.~\ref{uniqueness} guarantees that there is no ambiguity in the definition above (namely that $D$ has a unique type), and by Th.~\ref{typefixed} every fixed prime divisor 
in a Fano $4$-fold with $\rho_X\geq 6$ is of one of the four types defined above.

Let us also notice that by Prop.~\ref{fanotarget}, the definition given above of the type $(3,0)^{sm}$ is equivalent to Def.~\ref{30}.

The uniqueness property given by Cor.~\ref{uniqueness} does not hold without the assumption $\rho_X\geq 6$; here is a simple example (see also Ex.~\ref{special}).
\begin{example}\label{simple}
Set $X:=\pr_{\pr^1\times\pr^2}(\ma{O}\oplus\ma{O}(1,1))$. Then $X$ is a toric Fano $4$-fold with $\rho_X=3$, and the $\pr^1$-bundle $X\to\pr^1\times\pr^2$ has a section $D$ with normal bundle $\ma{N}_{D/X}\cong\ma{O}_{\pr^1\times\pr^2}(-1,-1)$. The divisor $D$ is the locus of two extremal rays $R_1$ and $R_2$ of $\NE(X)$, both divisorial, one of type $(3,1)^{sm}$ and one of type $(3,2)^{sm}$.
Thus in this case it does not make sense to talk about the ``type of the fixed divisor $D$''.
\end{example}
We conclude the section with some properties that will be used in the sequel.
\begin{remark}\label{davvero}
In case $(ii)$ of Th.~\ref{typefixed}, every non-trivial fiber $F$ of $f$ can intersect at most one exceptional line. Indeed if $F$ has two distinct points $p_1,p_2$ both belonging to an exceptional line, consider the line $C=\overline{p_1p_2}\subset F\cong \pr^2$, and its transform $C_X\subset X$. 
We have $-K_{\w{X}}\cdot C=2$ and $-K_X\cdot C_X\leq 0$ by Lemma \ref{SQM}(1), which is impossible.
\end{remark}
\begin{remark}\label{viapeyron}
In the setting of Th.~\ref{typefixed}, $\Exc(f)$ does not contain exceptional lines. Indeed this is clear in cases $(i)$, $(iii)$, and $(iv)$. In case $(ii)$, if $\Exc(f)$ contains an exceptional line $\ell$, then $f$ must be finite on $\ell$ (because $f$ is $K$-negative), thus $\ell$ intersects every non-trivial fiber of $f$. On the other hand by  Th.~\ref{typefixed} $\w{X}$ contains at least $\rho_X-4\geq 2$ exceptional lines, and these must all intersect $\Exc(f)$. Thus there is at least one non-trivial fiber of $f$ intersecting two exceptional lines, contradicting Rem.~\ref{davvero}.
\end{remark}
\begin{lemma}\label{dimension}
Let $X$ be a smooth Fano $4$-fold with $\rho_X\geq 6$, $D\subset X$ a fixed prime divisor not of type $(3,2)$, 
and $X\stackrel{\ph}{\dasharrow}\w{X}\stackrel{f}{\to}Y$ the MMP for $D$.
Let  $\ell_1,\dotsc,\ell_s\subset \w{X}$ be the  exceptional lines, and $H\subseteq\N(\w{X})$ the linear span of $[\ell_1],\dotsc,[\ell_s]$.
Then 
$$\dim H\geq\dim\N(D,X)-a\ \ \text{and}\ \ \dim \bigl(f_*H+\R[f(\Exc(f))]\bigr)=\dim\N(D,X)-1,$$
where $a=1$ if $f$ is of type $(3,0)$ and $a=2$ if $f$ is of type $(3,1)$.
\end{lemma}
\begin{proof}
By Lemma \ref{dimensions} we have 
$\dim\N(\Exc(f)\cup\ell_1\cup\cdots\cup\ell_s,\w{X})=\dim\N(D,X)$.
On the other hand $\N(\Exc(f)\cup\ell_1\cup\cdots\cup\ell_s,\w{X})=\N(\Exc(f),\w{X})+H$, thus
$$\dim\Bigl(\N(\Exc(f),\w{X})+H\Bigr)=\dim\N(D,X)$$
and hence
$\dim H\geq\dim\N(D,X)-\dim\N(\Exc(f),\w{X})$.
Using Rem.~\ref{basic}(1)-(2), this yields the inequality on $\dim H$.
Moreover,
clearly $\N(\Exc(f),\w{X})+H$ contains $\ker f_*$, and
$f_*(\N(\Exc(f),\w{X})+H)=f_*H+\R[f(\Exc(f))]$; this yields the 
second statement.
\end{proof}
\subsection{Fixed divisors of type $(3,1)^{sm}$}\label{a}
\noindent In this section we describe some properties of Fano $4$-folds having a fixed prime divisor of type $(3,1)^{sm}$. Our main goal is to show the following result, which will be used in section \ref{last}.
\begin{proposition}\label{(3,1)}
Let $X$ be a smooth Fano $4$-fold with $\rho_X\geq 7$.
If $X$ has a fixed prime divisor of type $(3,1)^{sm}$, then $X$  
 also has a fixed prime divisor of type $(3,2)$.
\end{proposition}
Let us fix a notation for the situation studied in this section.
\begin{notation}\label{not31}
Let $X$  be a smooth Fano $4$-fold with $\rho_X\geq 6$, $D\subset X$ a fixed prime divisor of type $(3,1)^{sm}$, and $X\stackrel{\ph}{\dasharrow}\widetilde{X}\stackrel{f}{\to}Y$
 the MMP for $D$, where $Y$ is smooth and Fano, and $f$ is the blow-up of the smooth curve $C\subset Y$.
\end{notation}
\begin{lemma}\label{firstprop31}
Assume that we are in Setting \ref{not31}.
Let $C_0\subset Y$ be an irreducible curve such that  $C\cap C_0\neq\emptyset$ and $C_0\neq C$. We have the following:
\begin{enumerate}[(1)]
\item
 $-K_Y\cdot C_0\geq 1$ and $-K_Y\cdot C_0\neq 2$;
\item 
if  $-K_Y\cdot C_0= 1$, then $C_0=f(\ell)$, $\ell\subset\w{X}$ an exceptional line with 
  $\Exc(f)\cdot \ell=1$; 
\item $C$ can meet (at most) finitely many curves of anticanonical degree $1$, and cannot meet any exceptional plane.
\end{enumerate}
\end{lemma}
\begin{proof}
The proof is analogous to that of Lemma \ref{locus}, see also \cite[proof of Cor.~3.19]{eff}.
\end{proof}
\begin{lemma}\label{31}
Assume that we are in Setting \ref{not31}, and that  $\rho_X\geq 7$. 
\begin{enumerate}[(1)]
\item
If $B\subset Y$ is a fixed prime divisor such that $B\cap C\neq\emptyset$, then $B$ is of type $(3,2)$.
\item If $B\subset Y$ is a fixed prime divisor containing $C$, then $B\cdot C<0$, $R:=\langle[C]\rangle$ is an extremal ray of type $(3,2)$ of $\NE(Y)$,  $R$ is the unique $B$-negative extremal ray in  $\NE(Y)$, and 
  $C$ is a one-dimensional fiber of the contraction of $R$.
\item One of the following holds:
\begin{enumerate}[$(i)$]
\item $[C]\in\mov(Y)$;
\item $[C]$ generates an extremal ray $R$ of $\NE(Y)$, of type $(3,2)$, such that $C$ is a one-dimensional fiber of the associated contraction; in particular 
 $C\cong\pr^1$, $-K_Y\cdot C=1$, and $\ma{N}_{C/Y}\cong \ma{O}_{\pr^1}(-1)\oplus\ma{O}_{\pr^1}^{\oplus 2}$.
\end{enumerate}
\end{enumerate}
\end{lemma}
\begin{proof}
Notice that $\rho_Y\geq 6$, hence the type of fixed prime divisors of $Y$ is well-defined. 
Let $B\subset Y$ be a fixed prime divisor intersecting $C$;
we assume by contradiction that $B$ is not of type $(3,2)$.
 Consider the MMP for $B$:
$$Y\stackrel{\psi}{\dasharrow}\widetilde{Y}\la Z$$
so that $\w{Y}\to Z$ is an elementary divisorial contraction with exceptional divisor the transform $\w{B}\subset\w{Y}$ of $B$.
 By Th.~\ref{typefixed}, $\psi$ factors as a sequence of at least $\rho_Y-4$ $K$-negative flips, hence $\w{Y}$ contains at least $\rho_Y-4\geq 2$ exceptional lines, which are pairwise disjoint and have positive intersection with 
 $\w{B}$.

Notice that $C\subset\dom\psi$ by Lemma \ref{firstprop31}(3). Let $p\in C\cap B$; we still denote by $p\in\w{B}$ its image. In particular $p\in\w{B}$ cannot belong to any exceptional line.
\begin{parg}
If $B$ is of type $(3,0)^Q$, let $\Gamma\subset\w{B}\cong Q$ be a general line through $p$. If instead $B$ is of type $(3,1)^{sm}$, let  $F\subset\w{B}$ be the fiber of  $\w{Y}\to Z$ containing $p$, and let  $\Gamma\subset F\cong\pr^2$ be a general line through $p$.
Then in both cases $-K_{\w{Y}}\cdot\Gamma=2$, and $\Gamma$ (being general) is distinct from the transform of $C$, and disjoint from the exceptional lines in $\w{Y}$. Therefore the transform $\w{\Gamma}\subset Y$ is an integral curve distinct from $C$, intersecting $C$, and with $-K_Y\cdot \w{\Gamma}=2$,  contradicting Lemma \ref{firstprop31}(1).
\end{parg}
\begin{parg}
Suppose now that $B$ is of type $(3,0)^{sm}$, and
let $q_1,\dotsc,q_h\in\w{B}$ be the points belonging to some exceptional line in $\w{Y}$, so that $h\geq 2$; notice that $q_i\neq p$ for every $i=1,\dotsc,h$. Consider the line $r_i:=\overline{pq_i}\subset\w{B}\cong\pr^3$ and its transform $\tilde{r}_i\subset B\subset Y$. Notice that $C\neq \tilde{r}_i$ for every $i=1,\dotsc,h$, because $C$ is disjoint from all exceptional planes in $Y$; on the other hand $p\in C\cap \tilde{r}_i$.

We have $-K_{\w{Y}}\cdot r_i=3$ and $-K_Y\cdot \tilde{r}_i\leq 2$ by Lemma \ref{SQM}(1) (because $r_i$ intersects some exceptional line). Since by Lemma \ref{firstprop31}(1) $C$ cannot intersect integral curves of anticanonical degree $2$ (distinct from $C$ itself), we must have $-K_Y\cdot \tilde{r}_i=1$ for every $i=1,\dotsc,h$. By Lemma \ref{firstprop31}(2), this implies that $\tilde{r}_i$ are images under $f\colon\w{X}\to Y$ of exceptional lines $\hat{r}_i$ in $\w{X}$. Since the curves $\tilde{r}_i$ in $Y$ all contain $p$, the curves $\hat{r}_i$ in $\w{X}$ all intersect the fiber $F_2\cong\pr^2$ of $f$ over $p$, but this is impossible
by Rem.~\ref{davvero} (as $h\geq 2$).

We conclude that $B$ is of type $(3,2)$, and we have (1).
\end{parg}
\begin{parg}
We show (2). If $B\subset Y$ is a fixed prime divisor containing $C$, then $B$ is of type $(3,2)$ by (1).
Let $R$ be an extremal ray of type $(3,2)$ with $\Lo(R)=B$, and let
$g\colon Y\to W$ be the associated contraction.
If $g(C)$ were a curve, $C$ should intersect infinitely many  curves of anticanonical degree $1$ (see \ref{basic0}); this contradicts Lemma \ref{firstprop31}(3). Therefore $C$ is contained in a fiber of $g$, and $R=\langle[C]\rangle$. This shows that $R$ is the unique $B$-negative extremal ray of $\NE(Y)$ of type $(3,2)$. By Rem.~\ref{pisa}(2), there can be no other $B$-negative extremal ray.

If $F$ is a $2$-dimensional fiber of $g$, then by \cite[Prop.~4.3.1]{AWaview} the possibilities for $(F,-K_{Y|F})$ are $(\pr^2,\ma{O}_{\pr^2}(1))$ or $(Q,\ma{O}_Q(1))$, where $Q\subset\pr^3$ is a reduced quadric. If $C$ were contained in such a fiber, it should intersect infinitely many curves of anticanonical  degree $1$, again contradicting Lemma \ref{firstprop31}(3). Thus $C$ is a $1$-dimensional fiber of $g$, and we have shown (2).
\end{parg}
\begin{parg}
 For (3), suppose that $[C]\not\in\mov(Y)$: since $\mov(Y)=\Eff(Y)^{\vee}$, there exists a prime divisor $B\subset Y$ such that $B\cdot C<0$. 
Since $C$ cannot be contained in an exceptional plane by Lemma \ref{firstprop31}(3), Rem.~\ref{movable} implies that $[B]\not\in\Mov(Y)$. Therefore $B$ is a fixed prime divisor containing $C$, and we get $(ii)$ by (2) (see also \ref{basic0}).
\end{parg}
\vspace{-20pt}
\end{proof}
\begin{lemma}\label{transform32}
Assume that we are in Setting \ref{not31}, and that $\rho_X\geq 7$.
Let $B_Y\subset Y$ be a fixed prime divisor such that $B_Y\cdot C>0$, and let $B_X\subset X$ be the transform of $B_Y$.
Then $B_X$ is a fixed prime divisor of type $(3,2)$ such that $D\cap B_X\neq\emptyset$, $\langle [D],[B_X]\rangle$ is a face of $\Eff(X)$, and 
  $\langle [D],[B_X]\rangle\cap\Mov(X)=\{0\}$.
\end{lemma}
\begin{proof}
 By Lemma \ref{bijection}(3), $B_X$ is a fixed prime divisor in $X$, and $\langle [D],[B_X]\rangle$ is a face of $\Eff(X)$ which intersects $\Mov(X)$ only in zero.

Since $B_Y\cdot C>0$, we have $B_Y\cap C\neq\emptyset$, thus $B_Y$ is a fixed divisor of type $(3,2)$
by Lemma \ref{31}(1). 
On the other hand
 $C\not\subset B_Y$ by Lemma \ref{31}(2), thus $C\cap B_Y$ is a finite set.

Let $B_{\w{X}}\subset\w{X}$ be the transform of $B_Y$. Then $\dim(B_{\w{X}}\cap\Exc(f))=2$, and since $\dim(\w{X}\smallsetminus\dom\ph^{-1})=1$ by Th.~\ref{SQM0}, we have 
 $B_X\cap D\neq\emptyset$.

We are left to show that $B_X$ is of type $(3,2)$.
Let $g\colon Y\to Z$ be an elementary contraction of type $(3,2)$ with $\Exc(g)=B_Y$. Then $g^{-1}(g(C\cap B_Y))$ is a union of finitely many non-trivial fibers of $g$, and $U_Y:=B_Y\smallsetminus g^{-1}(g(C\cap B_Y))$ is a non-empty open subset of $B_Y$, covered by proper, integral rational curves of anticanonical degree $1$ (see \ref{basic0}).
Since $U_Y\cap C=\emptyset$, the transform $U_{\w{X}}\subset\w{X}$ of $U_Y$ is isomorphic to $U_Y$, and by Lemma \ref{SQM}(3) we have $U_{\w{X}}\subset\dom\ph^{-1}$. In conclusion, the transform $U_X\subset B_X$ of $U_Y$ in $X$ is contained in the open subset where the map
$X\dasharrow Y$ is an isomorphism.

Since $X$ is Fano, this implies that $B_X$ is covered by a family of rational curves of anticanonical degree $1$, and $B_X$ is not nef. We conclude that $B_X$ is of type $(3,2)$ by Lemma \ref{divisorcovered}.
\end{proof}
\begin{lemma}\label{lungo}
Assume that we are in Setting \ref{not31}.
Let $\sigma$ be a 
 one-dimensional face of $\Eff(Y)$ such that $\sigma\cdot C>0$. Then
$\sigma\not\subset\Mov(Y)$.
\end{lemma}
\begin{proof}
We proceed by contradiction, and assume that there exists a  one-dimensional face $\sigma$ of $\Eff(Y)$ such that $\sigma\cdot C>0$ and
$\sigma\subset\Mov(Y)$. As in the proof of Lemma \ref{useful}, this yields
a diagram
$$\xymatrix{Y\ar@{-->}[r]\ar@/^1pc/@{-->}[rr]^g&{\w{Y}}\ar[r]_{\tilde{g}}&{Z,}
}$$
where $g$ is a  rational contraction with $\sigma=g^*(\Nef(Z))$, $Y\dasharrow \w{Y}$ is a SQM, $\w{Y}$ is smooth, and $\tilde{g}$ is a $K$-negative contraction. Moreover if  $F\subset\w{Y}$ is a general fiber of $\tilde{g}$, then $F$ is smooth, Fano, and 
$\rho_F\geq\rho_Y-\dim\sigma=\rho_Y-1\geq 4$ by Lemma \ref{movf}, 
thus $g$ is of fiber type and $F$ has dimension $2$ or $3$. 

In particular, since $\rho_F\geq 4$,  $F$ has a covering family of rational curves of anticanonical degree $2$. This is well-known if $F$ is a del Pezzo surface, and follows from Mori and Mukai's work if $\dim F=3$, see \cite[Theorem on p.~141]{fanoEMS}.
Therefore $\w{Y}$ has a covering family $V$ of rational curves, of anticanonical degree $2$, and contracted by $\tilde{g}$.

 Notice that $C$ cannot meet any exceptional plane in $Y$ by Lemma \ref{firstprop31}(3), therefore it is contained in the open subset where the map $Y\dasharrow \w{Y}$ is an isomorphism; let $\w{C}\subset\w{Y}$ be the transform of $C$. Since $\sigma\cdot C>0$, the image $\tilde{g}(\w{C})$ is a curve.

For every point $p\in \w{C}$ there exists a curve $C_{v_p}$ of the family $V$ containing $p$. Let $\Gamma_p$ be an irreducible component of $C_{v_p}$ containing $p$, so that $\Gamma_p$ is  contracted by $\tilde{g}$ (hence $\Gamma_p\neq\w{C}$), and $1\leq -K_{\w{Y}}\cdot\Gamma_p\leq 2$.
Notice also that when $p$ varies in $\w{C}$, we get infinitely many distinct curves $\Gamma_p$.

Let $\Gamma_p'\subset Y$ be the transform of $\Gamma_p$. Then $\Gamma_p'$ still meets $C$, and  $1\leq -K_{{Y}}\cdot\Gamma_p'\leq 2$ by Lemma \ref{SQM}(1). Thus $-K_{{Y}}\cdot\Gamma_p'=1$ by Lemma \ref{firstprop31}(1), therefore $C$ intersects infinitely many curves of anticanonical degree $1$, contradicting 
Lemma \ref{firstprop31}(3).
\end{proof}
\begin{proof}[Proof of Prop.~\ref{(3,1)}]
Let $D\subset X$ be a fixed prime divisor of type $(3,1)^{sm}$, and $X\dasharrow\widetilde{X}\stackrel{f}{\to}Y$ 
 the MMP for $D$,
where $Y$ is a smooth Fano $4$-fold and $f$ is the blow-up of a smooth curve $C\subset Y$.
Since $C$ has positive intersection with some effective divisor in $Y$,
there exists a one-dimensional face $\sigma$ of $\Eff(Y)$ such that $C\cdot\sigma>0$.
By Lemma \ref{lungo}, we have
 $\sigma\not\subset\Mov(Y)$, thus $\sigma$ contains the class of a fixed prime divisor $B\subset Y$ such that $B\cdot C>0$.  Then Lemma \ref{transform32} yields the statement.
\end{proof}
\subsection{Fixed divisors of type $(3,0)^Q$}\label{b}
\noindent In this section we describe some properties of Fano $4$-folds having a fixed prime divisor $D$ of type $(3,0)^Q$. 
Our main goal is to show the following result, which will be used in section \ref{last}.
\begin{proposition}\label{(3,0)Q}
Let $X$ be a smooth Fano $4$-fold with $\rho_X\geq 6$, and suppose that $X$ has a fixed prime divisor of type $(3,0)^Q$. Then one of the following holds:
\begin{enumerate}[$(i)$]
\item  $X$ contains  a prime divisor covered by rational curves of anticanonical degree $1$; 
\item $X$ has a 
covering family of rational curves of anticanonical degree $3$.
\end{enumerate}
\end{proposition}
Let us fix a notation for the situation studied in this section.
\begin{notation}\label{not30Q}
Let $X$  be a smooth Fano $4$-fold with $\rho_X\geq 6$, $D\subset X$ a fixed prime divisor of type $(3,0)^Q$, and $X\stackrel{\ph}{\dasharrow}\widetilde{X}\stackrel{f}{\to}Y$
 the MMP for $D$, where $f$ is of type $(3,0)^Q$ and $Y$ is Fano.
 Set $p:=f(\Exc(f))\in Y$; we have $\Sing(Y)=\{p\}$ and $p$ is a factorial, terminal, l.c.i.\ singularity (see  Lemma \ref{quadricsing}).
\end{notation}
\begin{lemma}
\label{firstpropQ}
Assume that we are in Setting \ref{not30Q}.
Let $C\subset Y$ be an integral curve containing $p$. Then we have the following:
\begin{enumerate}[(1)]
\item $-K_Y\cdot C\geq 1$ and $-K_Y\cdot C\neq 2$;
\item
if $-K_Y\cdot C=1$, then $C=f(\ell)$, $\ell\subset \w{X}$ an exceptional line with $\Exc(f)\cdot\ell=1$;
\item $p$ belongs to (at most) finitely many curves of anticanonical degree $1$.
\end{enumerate}
\end{lemma}
\begin{proof}
The proof is analogous to that of Lemma \ref{locus}.
\end{proof}
\begin{lemma}\label{smallQ}
Assume that we are in Setting \ref{not30Q}.
\begin{enumerate}[(1)]
\item
If $R$ is a small extremal ray of $\NE(Y)$, then $p\not\in\Lo(R)$, and 
$f(\ell)\cap \Lo(R)=\emptyset$
for every exceptional line $\ell\subset\w{X}$.
\item
If $\psi\colon Y\dasharrow\w{Y}$ is a SQM, then $p\in\dom\psi$, and Th.~\ref{SQM0} holds for $\psi$ even if $Y$ has a singular point. Moreover $f(\ell)\subset\dom\psi$
for every exceptional line $\ell\subset\w{X}$.
\end{enumerate}
\end{lemma}
\begin{proof}
Let $R$ be a small extremal ray of $\NE(Y)$. Recall that $Y$ is Fano, hence $R$ is $K$-negative. By   Lemma \ref{quadricsing} $Y$ has an isolated terminal l.c.i.\ singularity at $p$. By the works of Kawamata, Andreatta and Wi\'sniewski, and Kachi (see \cite[Th.~1.1, Cor.~2.2]{kachiflips} and references therein), for every irreducible component $T$ of $\Lo(R)$ we have $(T,-K_{Y|T})\cong(\pr^2,\ma{O}_{\pr^2}(1))$. Therefore $p\not\in T$, because $T$ cannot belong to infinitely many curves of anticanonical degree $1$ by Lemma \ref{firstpropQ}(3). 

We conclude that $p\not\in\Lo(R)$, thus $\Lo(R)$ is contained in the smooth locus of $Y$, and the locus of $R$ and its flip are described in \cite{kawsmall}; in particular $\Lo(R)$ is a finite disjoint union of exceptional planes.

Hence the transform of $\Lo(R)$ in $\w{X}$ is again a union of exceptional planes. If $\ell\subset\w{X}$ is an exceptional line, then $\ell$ cannot meet any exceptional plane by Lemma \ref{SQM}(3), hence $f(\ell)\cap\Lo(R)=\emptyset$. This gives (1).

For (2), the same proof as the one of \cite[Rem.~3.6]{eff} works: since $Y$ is Fano, $\psi$ can be factored as a sequence of $K$-negative flips; the locus of each flip is contained in the smooth part by (1), and the loci of the flips are all disjoint.
\end{proof}
\begin{lemma}\label{partial(3,0)Q}
Assume that we are in Setting \ref{not30Q}.
\begin{enumerate}[(1)]
\item
If there exists an irreducible curve $C\subset Y$, containing $p$, with $-K_Y\cdot C\leq 2$, then $X$ contains  a prime divisor covered by rational curves of anticanonical degree $1$.
\item If $-K_Y\cdot C\geq 3$ for every irreducible curve $C\subset Y$ containing $p$, then
$X$  has a 
covering family of rational curves of anticanonical degree $3$.
\end{enumerate}
\end{lemma}
\begin{proof}
We know by Lemma \ref{quadricsing} that $\w{D}:=\Exc(f)$ is either a smooth quadric, or a cone over a smooth $2$-dimensional quadric, and that $Y$ has a l.c.i.\ singularity at $p$.
\begin{parg}
We suppose first that there exists an irreducible curve $C\subset Y$ such that $p\in C$ and $-K_Y\cdot C\leq 2$.
Then by Lemma \ref{firstpropQ}(1)-(2) we know that $-K_Y\cdot C=1$ and that $C=f(\ell)$, where $\ell\subset\w{X}$ is an exceptional line with $\w{D}\cdot\ell=1$. In particular $\ell$ intersects $\w{D}$ transversally at a smooth point $q$. 

The exceptional line $\ell$ corresponds, under $\ph\colon X\dasharrow \w{X}$, to an exceptional plane $L\subset D$, such that $D$ is smooth along $L$ and the birational map $\ph_{|D}\colon D\dasharrow\w{D}$ realizes locally $D$ as the blow-up of $\w{D}$ at $q$, with exceptional divisor $L$.

Let $r\subset \w{D}$ be a curve corresponding (under the isomorphism $\w{D}\cong Q$) to a general
line in $Q$ through $q$, and let $r_X\subset D$ be its transform in $X$. 
We have $-K_{\w{X}}\cdot r=2$, and $r$ meets $\ell$ transversally at $q$.
By Lemma \ref{SQM}(1) $q$ is the unique point of $r$ belonging to an
 exceptional line, and $-K_X\cdot r_X=1$; moreover $r_X\cong r\cong\pr^1$ and $r_X$ is contained in the smooth locus of $D$.

In the blow-up of $\w{D}\cong Q$ at the smooth point $q$ we have $\ma{N}_{r_X/D}\cong\ma{O}_{\pr^1}\oplus\ma{O}_{\pr^1}(-1)$; in particular $-K_D\cdot r_X=1$, and by adjunction $D\cdot r_X=0$.
Thus $\ma{N}_{D/X|r_X}\cong\ma{O}_{\pr^1}$, and we get an exact sequence of sheaves on $\pr^1$:
$$0\la\ma{O}_{\pr^1}\oplus\ma{O}_{\pr^1}(-1)\la\ma{N}_{r_X/X}\la\ma{O}_{\pr^1}
\la 0.$$
Since $\text{Ext}^1(\ma{O}_{\pr^1},\ma{O}_{\pr^1}\oplus\ma{O}_{\pr^1}(-1))=0$, we deduce that
 $\ma{N}_{r_X/X}\cong \ma{O}_{\pr^1}^{\oplus 2}\oplus\ma{O}_{\pr^1}(-1)$.

Therefore $h^1(r_X,\ma{N}_{r_X/X})=0$ and $h^0(r_X,\ma{N}_{r_X/X})=2$. 
 This implies that 
both $\text{Hilb}(X)$ and $\text{Chow}(X)$ are smooth at $[r_X]$, and locally of dimension $2$ \cite[Th.~I.2.8 and Cor.~I.6.6.1]{kollar}.
 Let $V$ be the irreducible component of $\text{Chow}(X)$ containing $[r_X]$. Then $\dim V=2$, $\Lo(V)$ is an irreducible closed subset with $\dim\Lo(V)\leq 3$, and since $X$ is Fano and $-K_X\cdot r_X=1$, $C_v$ is an integral rational curve for every $v\in V$.

Let $S\subset D$ be the strict transform of the union of lines in $\w{D}\cong Q$ through $q$. Then $\dim S=2$ and $S\subseteq \Lo(V)$. Moreover $S$ intersects the exceptional plane $L$ in a conic $\Gamma$, and $S$ contains the curve $r_X$.
We have  $D\cdot \Gamma<0$ and $D\cdot r_X=0$, thus
 $\dim\N(S,X)\geq 2$.

If $\dim\Lo(V)=2$, then for any $x\in\Lo(V)$ there are infinitely many curves of the family $V$ through $x$. This implies that $\Lo(V_x)=\Lo(V)=S$, but this is impossible because $\N(\Lo(V_x),X)
=\R[V]$ by Lemma \ref{kol}, while  $\dim\N(S,X)\geq 2$.
Therefore $\Lo(V)$ is a prime divisor, and we have (1).
\end{parg}
\begin{parg}
Suppose now that $-K_Y\cdot C\geq 3$ for every irreducible curve $C\subset Y$ containing $p$.
Let $W$ be a maximal covering family of rational curves  in $Y$, of minimal anticanonical degree among such families. Since some curve of the family must contain $p$, we have $-K_Y\cdot [W]\geq 3$.

By (the proof of) \cite[Th.~IV.2.4]{kollar} the family $W$ is ``generically unsplit'' in the terminology of \cite{kollar}. This implies that for $y\in Y$ general, the curve $C_w$ is integral for every $w\in W_y$, therefore $\dim\N(\Lo(W_y),Y)=1$ by Lemma \ref{kol}. 

Since $\rho_Y\geq 5$,  by \cite[Rem.~5.5]{gloria} we have  $\dim\N(B,Y)\geq 2$ for every prime divisor $B\subset Y$, hence we must have $\dim\Lo(W_y)\leq 2$.

Even if $Y$ is singular, it has l.c.i.\ singularities, and $\Sing(Y)$ is a point; thus \cite[Th.~II.1.3]{kollar} yields $\dim W\geq -K_Y\cdot [W]+1$. We can apply
this to get \cite[Cor.~IV.2.6]{kollar} as in the smooth case, and we get $-K_Y\cdot [W]\leq \dim\Lo(W_y)+1\leq 3$, thus
$-K_Y\cdot [W]=3$.

We also know that the curve $C_w$ is integral for every $w\in W_p$, thus as before we deduce that $\dim\Lo(W_p)\leq 2$, in particular $W_p\subsetneq W$. This means that there is a non-empty open subset $W^0\subset W$ such that for every $w\in W^0$
the curve $C_w$ is an integral rational curve contained in the smooth locus $Y_{reg}=Y\smallsetminus\{p\}$. 

Let $w\in W^0$ and let $f_w\colon\pr^1\to Y_{reg}$ be a morphism given by the normalization of $C_w$. Then, for general $w\in W^0$, $f_w$ is a free rational curve in $Y_{reg}$. Hence by \cite[Prop.~II.3.7]{kollar}, for general $w\in W^0$, the curve $C_w$ is disjoint from 
 the images in $Y$ of exceptional curves in $\w{X}$, namely $C_w$ is contained in the open subset where the map $X\dasharrow Y$ is an isomophism. 
Therefore $X$ has a covering family of rational curves of anticanonical degree $3$, and we have (2).
\end{parg}
\vspace{-20pt}
\end{proof}
\begin{proof}[Proof of Prop.~\ref{(3,0)Q}]
It is a straightforward consequence of Lemma \ref{partial(3,0)Q}.
\end{proof}
\subsection{Elementary divisorial rational contractions}\label{dual}
\noindent Let $X$ be a Fano $4$-fold with $\rho_X\geq 7$, or $\rho_X=6$ and $\delta_X\leq 2$. We have defined in section \ref{tf} the type of a fixed prime divisor $D\subset X$, in terms of the unique MMP for $D$. In this section we show that whenever $D$ becomes the exceptional divisor of an elementary divisorial contraction on a SQM of $X$, this contraction is of the same  type as $D$, and contracts  the same class of curves from $X$ (Corollaries \ref{vialechiuse} and \ref{portovenere}). This is false in general for smaller values of $\rho_X$, and generalizes to arbitrary fixed prime divisors  some properties that we have seen in section \ref{every1} for type $(3,0)^{sm}$ (and any $\rho_X$), in particular  Prop.~\ref{tofix}. 
\begin{lemma}\label{elementarydivisorial}
Let $X$ be a smooth Fano $4$-fold with  $\rho_X\geq 7$, or $\rho_X=6$ and $\delta_X\leq 2$, and 
 $D\subset X$ a fixed prime divisor. Let
 $\psi\colon X\dasharrow X'$ be a SQM, and $g\colon X'\to Y'$ an elementary divisorial contraction with exceptional divisor the transform of $D$.
Let  $X\stackrel{\ph}{\dasharrow}\w{X}\stackrel{f}{\to}Y$ be the MMP for $D$, and set $\xi:=\psi\circ\ph^{-1}$ and $\xi_Y:=g\circ\xi\circ f^{-1}$.
$$\xymatrix{X \ar@{-->}[r]_{\ph}\ar@/^1pc/@{-->}[rr]^{\psi} &{\w{X}}\ar[d]_f\ar@{-->}[r]_{\xi}&{X'}\ar[d]^{g}\\
&Y\ar@{-->}[r]^{\xi_Y}&{Y'}}
$$
Then $\Exc(f)\subset\dom\xi$ and $f(\Exc(f))\subset\dom\xi_Y$, so that $f$ and $g$ are locally isomorphic. 
\end{lemma}
\begin{proof}
If $D$ is of type $(3,0)^{sm}$ this follows Prop.~\ref{tofix} (and its proof), and if $D$ is of type $(3,1)^{sm}$ this follows from (the proof of)
\cite[Cor.~3.19]{eff}. 

Suppose that $D$ is of type $(3,0)^Q$, and set $p:=f(\Exc(f))\in Y$. Notice that $\xi_Y$ is birational and an isomorphism in codimension $1$, that is, a SQM. 
Then $p\in\dom\xi_Y$ by Lemma \ref{smallQ}(2), and the statement follows as in the proof of Prop.~\ref{tofix}.

Finally, suppose that $D$ is of type $(3,2)$, and recall that in this case $\ph$ is an isomorphism (see Th.~\ref{typefixed}), so we can assume that $X=\w{X}$ and $f\colon X\to Y$.
Notice that by our assumptions on $\rho_X$ and $\delta_X$, Rem.~\ref{pisa}(3) yields that $\NE(f)$ is the unique $D$-negative extremal ray of $\NE(X)$, thus $\NE(f)=(-K_X+D)^{\perp}\cap\NE(X)$ by Lemma \ref{torino}(1).

If $\psi$ is an isomorphism, the statement follows from Cor.~\ref{uniqueness}.

Let us assume that $\psi$ is not an isomorphism. In particular $X$ is not a product of  surfaces, thus $\delta_X\leq 3$ by Th.~\ref{summary}.
By  Rem.~\ref{pisa}(2) $D$ does not contain exceptional planes, and \cite[Cor.~3.14]{eff} yields $\dim\N(\Exc(g),X')=\dim\N(D,X)\geq \rho_X-\delta_X\geq 4$. Therefore $g$ is of type $(3,2)$ (see Rem.~\ref{basic}(1)-(2))
and $\Exc(g)$ is covered by curves of anticanonical degree $1$ (see \ref{basic0}), thus $\Exc(g)\subset\dom(\psi^{-1})$ and $D\subset\dom\psi$ by Lemma \ref{SQM}(3).

Let $F\subset\Exc(g)$ be a general fiber of $g$, and $F_X\subset D$ its transform. Since $(-K_{X'}+\Exc(g))\cdot F=0$ in $X'$ (see \ref{basic0}), we have $(-K_X+D)\cdot F_X=0$ in $X$.
On the other hand
$\NE(f)=(-K_X+D)^{\perp}\cap\NE(X)$,
 therefore $[F_X]\in \NE(f)$; this implies the statement. 
\end{proof}
The following is a straightforward consequence of
 Lemma \ref{elementarydivisorial} and Th.~\ref{typefixed}.
\begin{corollary}\label{vialechiuse}
Let $X$ be a smooth Fano $4$-fold with $\rho_X\geq 7$, or $\rho_X=6$ and $\delta_X\leq 2$.
Let  $\psi\colon X\dasharrow X'$ be a SQM and $g\colon X'\to Y'$ an elementary divisorial contraction.
Then $g$ is  either of type $(3,0)^{sm}$, or $(3,1)^{sm}$, or $(3,2)$, or $(3,0)^Q$.
Moreover, let $D\subset X$ be the transform of $\Exc(g)\subset X'$, so that $D$ is a fixed prime divisor. Then $g$ and $D$ have the same type.
\end{corollary}
 We associate to every fixed prime divisor $D$ a smooth rational curve, moving in a family which dominates $D$. When $D$ is of type $(3,0)^{sm}$, this  is the same as Def.~\ref{prima}.
\begin{definition}[the curve $C_D$ moving in a fixed divisor $D$]
Let $X$ be a smooth Fano $4$-fold with $\rho_X\geq 7$, or $\rho_X=6$ and $\delta_X\leq 2$, and let $D\subset X$
be a fixed prime divisor. 
Let $X\stackrel{\ph}{\dasharrow}\widetilde{X}\stackrel{f}{\la}Y$
be the MMP for $D$ (unique by Cor.~\ref{uniqueness}).
 Let $C_D\subset D\subset X$ be, respectively:
\begin{enumerate}[--]
\item 
if $D$ is of type $(3,0)^{sm}$ or $(3,0)^Q$: the transform of a general line in $\Exc(f)$;
\item
if $D$ is of type $(3,1)^{sm}$: the transform of a line in a general fiber of  $f$;
\item if $D$ is of type $(3,2)$ (so that $\ph$ is an isomorphism): a general non-trivial fiber of $f$.
\end{enumerate}
We note that $C_D\cong\pr^1$ and $D\cdot C_D=-1$ for every $D$, and that $C_D$ moves in a family covering a dense subset of $D$. Moveover $C_D$ is contained in the open subset where $\ph\colon X\dasharrow\w{X}$ is an isomorphism. Finally we have $-K_X\cdot C_D=1$ if $D$ is of type $(3,2)$, $-K_X\cdot C_D=2$ if $D$ is of type $(3,1)^{sm}$ or $(3,0)^Q$, and  $-K_X\cdot C_D=3$ if $D$ is of type $(3,0)^{sm}$.
\end{definition}
The following corollary follows immediately from 
 Lemma \ref{elementarydivisorial}.
\begin{corollary}\label{portovenere}
Let $X$ be a smooth Fano $4$-fold with $\rho_X\geq 7$, or $\rho_X=6$ and $\delta_X\leq 2$.
Let  $\psi\colon X\dasharrow X'$ be a SQM and $g\colon X'\to Y'$ an elementary divisorial contraction. Let $D\subset X$ be the transform of $\Exc(g)\subset X'$, so that $D$ is a fixed prime divisor.
Then $C_D\subset\dom\psi$ and $g(\psi(C_D))=\{pt\}$.
\end{corollary}
\subsection{The dual of the cone of movable divisors}
\noindent Recall that we have the following cones of curves in $\N(X)$:
$$\mov(X)\subseteq \Mov(X)^{\vee}\subseteq\NE(X).$$
It has been shown by Choi \cite{choi} that in a Mori dream space $X$, the cone 
 $\Mov(X)^{\vee}$ is generated by the classes of families of
curves covering at least a divisor, on $X$  or on a SQM of $X$ (we refer the reader to \cite{choi} for a precise statement). In our situation, thanks to Cor.~\ref{portovenere}, we do not need to consider SQM's, and we can describe 
$\Mov(X)^{\vee}$ simply in terms of $\mov(X)$ and of the curves $C_D$ (see also Lemma \ref{movdual2}).
\begin{lemma}\label{movdual}
Let $X$ be a smooth Fano $4$-fold with $\rho_X\geq 7$, or $\rho_X=6$ and $\delta_X\leq 2$. 
\begin{enumerate}[(1)]
\item $\langle[C_D]\rangle$ is a one-dimensional face of $\Mov(X)^{\vee}$, not contained in $\mov(X)$,
 for every fixed prime divisor $D$;

\smallskip

\item
$\Mov(X)=\bigl(\langle[C_D]\rangle_{D\text{\,fixed}}\bigr)^{\vee}\cap\Eff(X)\ $ and $\ \Mov(X)^{\vee}=\langle[C_D]\rangle_{D\text{\,fixed}}+\mov(X)$.
\end{enumerate}
\end{lemma}
\begin{proof}
We first show (2).
Notice that the two equalities are dual to each other, so it is enough to prove the first one.

The inclusion $\subseteq$ is easy. Indeed $\Mov(X)$ is generated by the classes of movable prime divisors, and if $B_0$ is such a divisor, then $B_0\neq D$ for every fixed prime divisor $D$, so that $B_0\cdot C_D\geq 0$.

Consider now an effective divisor $B$ such that $B\cdot C_D\geq 0$ for every fixed prime divisor $D\subset X$, and let us run a MMP for $B$. Since $B$ is effective,  the MMP must end with $B$ becoming nef.

We show that the MMP cannot have divisorial contractions. Otherwise, by considering the first divisorial contraction, we get a diagram
$$X\stackrel{\psi}{\dasharrow} X'\stackrel{g}{\la} Y'$$
where $\psi$ is a SQM, and $g\colon X'\to Y'$ is a $B$-negative elementary divisorial contraction. 
Let $D\subset X$ be the transform of $\Exc(g)$. Then $D$ is a fixed prime divisor, and by Cor.~\ref{portovenere} we have $C_D\subset\dom\psi$ and $[\psi(C_D)]\in\NE(g)$. We have
 $B\cdot C_D\geq 0$ in $X$, contradicting the fact that $g$ is $B$-negative. 
We conclude that the MMP yields a SQM $X\dasharrow X'$ such that the transform of $B$ in $X'$ is nef, and hence $B$ is movable.
Therefore we have the equality
$$\Mov(X)=\bigl(\langle[C_D]\rangle_{D\text{\,fixed}}\bigr)^{\vee}\cap\Eff(X).$$

We show (1); let $D_0$ be a fixed prime divisor. Then we have $D_0\cdot C_{D_0}<0$ and $D_0\cdot C_D\geq 0$ for every fixed prime divisor $D\neq D_0$, therefore
$$[D_0]\in \bigl(\langle[C_D]\rangle_{D\text{\,fixed},D\neq D_0}\bigr)^{\vee}
\cap\Eff(X)\quad\text{and }[D_0]\not\in\Mov(X),$$
hence 
$$\bigl(\langle[C_D]\rangle_{D\text{\,fixed},D\neq D_0}\bigr)^{\vee}
\cap\Eff(X)\supsetneq\Mov(X)$$
and dually
$$\langle[C_D]\rangle_{D\text{\,fixed},D\neq D_0}+\mov(X)\subsetneq\Mov(X)^{\vee}.$$
This shows that $\langle[C_{D_0}]\rangle$ is a one-dimensional face of $\Mov(X)^{\vee}$. Moreover $C_{D_0}\cdot D_0<0$, thus $[C_{D_0}]\not\in\mov(X)$.
\end{proof}
\begin{lemma}\label{movdual2}
Let $X$ be a smooth Fano $4$-fold with $\rho_X\geq 7$, or $\rho_X=6$ and $\delta_X\leq 2$.
Then one of the following holds:
\begin{enumerate}[$(i)$]
\item  $\rho_X\leq 11$ and $X$ has an elementary rational contraction of fiber type;

\smallskip

\item
$\Mov(X)^{\vee}=\langle[C_D]\rangle_{D\text{\,fixed}}$.
\end{enumerate}
\end{lemma}
\begin{proof}
We know by
Lemma \ref{movdual} that
$$\Mov(X)^{\vee}=\langle[C_D]\rangle_{D\text{\,fixed}}+\mov(X)\supseteq \langle[C_D]\rangle_{D\text{\,fixed}}.$$
Suppose that $(ii)$ does not hold, so that
 $\Mov(X)^{\vee}$ is strictly bigger than the cone $\langle[C_D]\rangle_{D\text{\,fixed}}$. Then  there exists a one-dimensional face $\sigma$ of $\mov(X)$ which is also a one-dimensional face of $\Mov(X)^{\vee}$. Dually, this means that the hyperplane $\sigma^{\perp}$ cuts both a facet of $\Mov(X)$ and $\Eff(X)$. By \cite[Cor.~2.25]{eff} this implies that $X$ has an elementary rational contraction of fiber type, and  $\rho_X\leq 11$ by Th.~\ref{elemft}, so we have $(i)$.
\end{proof}
In the setting of Lemma \ref{movdual}, a fixed prime divisor $D\subset X$ determines a facet of $\Mov(X)$, given by $\Mov(X)\cap C_D^{\perp}$. For a rational contraction $g\colon X\dasharrow Z$, consider the corresponding cone 
$\sigma:=g^*(\Nef(Z))\in\MCD(X)$. The following lemma describes 
the geometrical meaning for $g$ 
 of the condition $\sigma\subseteq\Mov(X)\cap C_D^{\perp}$.
\begin{lemma}\label{treno}
Let $X$ be a smooth Fano $4$-fold with $\rho_X\geq 7$, or $\rho_X=6$ and $\delta_X\leq 2$. Let $D\subset X$ be a fixed prime divisor, and consider the corresponding curve $C_D\subset D$.
\begin{enumerate}[(1)]
\item
Let $f\colon X\dasharrow Y$ be an elementary rational contraction. Then $f^*(\Nef(Y))\subset C_D^{\perp}$ if and only if $f$ is divisorial with exceptional divisor $D$.
\item
Let $f_0\colon X\dasharrow Y_0$ be an elementary divisorial rational contraction with exceptional divisor $D$, and let
 $g\colon X\dasharrow Z$ be a rational contraction. Then $g^*(\Nef(Z))\subset C_D^{\perp}$ if and only if $g\circ f_0^{-1}\colon Y_0\dasharrow Z$ is a rational contraction.
\end{enumerate}
\end{lemma}
\begin{proof}
We show (1).
If  $f$ is divisorial with exceptional divisor $D$, then $C_D\subset\dom f$ and $f(C_D)$ is a point by Cor.~\ref{portovenere}, therefore $f^*(\Nef(Y))\subset C_D^{\perp}$.

Conversely, assume that $f^*(\Nef(Y))\subset C_D^{\perp}$. 
We know by Lemma \ref{movdual}(1) that $\langle [C_D]\rangle$ is a one-dimensional face of $\Mov(X)^{\vee}$, not contained in $\mov(X)$. Dually, this means that the hyperplane  $C_D^{\perp}$ cuts a facet of $\Mov(X)$, and not of $\Eff(X)$. Thus 
 $f^*(\Nef(Y))$ is contained in $\partial\Mov(X)$ and not in $\partial\Eff(X)$, which means that $f$ is divisorial. By the first part of the proof, $f^*(\Nef(Y))\subset C_{\Exc(f)}^{\perp}$. This implies that $D=\Exc(f)$.

We show (2). 
If $h:=g\circ f_0^{-1}\colon Y_0\dasharrow Z$ is a rational contraction, then $g^*(\Nef(Z))=f_0^*(h^*(\Nef(Z))\subset f_0^*(\Nu(Y_0))=C_D^{\perp}$.

Conversely, set  $\sigma:=g^*(\Nef(Z))\in\MCD(X)$, and assume that
$\sigma\subset C_D^{\perp}$. Then $\sigma$ is contained in $\Mov(X)\cap C_D^{\perp}$, which is a facet of $\Mov(X)$ by Lemma \ref{movdual}(1).
Let   $\tau\in\MCD(X)$ be such that $\sigma\subseteq\tau\subseteq\Mov(X)\cap C_D^{\perp}$ and $\dim\tau=\rho_X-1$. Let $f_1\colon X\dasharrow Y_1$ be the rational contraction such that $\tau=f_1^*(\Nef(Y_1))$; notice that $f_1$ is elementary, and by (1) it is divisorial with exceptional divisor $D$.
$$\xymatrix{&X\ar@{-->}[dl]_{f_0}\ar@{-->}[d]^{f_1}\ar@{-->}[dr]^g&\\
{Y_0}\ar@{-->}@/_1pc/[rr]_{g\circ f_0^{-1}}\ar@{-->}[r]&{Y_1}\ar[r]&Z
}$$
Since $\sigma\subseteq\tau$, the composition $g\circ f_1^{-1}\colon Y_1\to Z$ is regular (see for instance \cite[Lemma 1.7]{hukeel} or \cite[p.~592]{eff}). Moreover the composition $f_1\circ f_0^{-1}\colon Y_0\dasharrow Y_1$ is an isomorphism in codimension $1$, and hence a SQM. Finally, $g\circ f_0^{-1}\colon Y_0\dasharrow Z$ is the composition of a SQM with a regular contraction, therefore it is a rational contraction.
\end{proof}
\subsection{Fixed divisors of type $(3,2)$}
\noindent In this section we conclude our analysis of the different types of fixed  divisors, by considering the case of type $(3,2)$: we show the following result, which  will be crucial in section \ref{last}.
\begin{proposition}\label{zero}
Let $X$ be a smooth Fano $4$-fold with $\rho_X\geq 12$ and $\delta_X\leq 1$. Then one of the following holds:
\begin{enumerate}[$(i)$]
\item $\rho_X=12$ and there exists $f\colon X\to X_1$ 
where $X_1$ is smooth and Fano, $f$ is the blow-up of a smooth, irreducible surface,  and $X_1$ has an elementary rational contraction onto a $3$-fold;  
\item for every fixed prime divisor $D\subset X$ of type $(3,2)$, we have  $\N(D,X)=\N(X)$.
\end{enumerate}
\end{proposition}
\begin{proof}
Suppose that $(ii)$ does not hold, and let $D_0\subset X$ be a fixed prime divisor of type $(3,2)$ with $\N(D_0,X)\subsetneq\N(X)$. Since $\delta_X\leq 1$, the subspace $\N(D_0,X)$ is a hyperplane in $\N(X)$.
\begin{parg}\label{coop}
We have 
 $\Mov(X)^{\vee}=\langle [C_D]\rangle_{D\text{ fixed}}$ by Lemma \ref{movdual2}.
Consider the class $[C_{D_0}]$ in $\Mov(X)^{\vee}$, and let $D_1,\dotsc,D_r$ be all the fixed prime divisors such that $\langle [C_{D_0}],[C_{D_i}]\rangle$ is a face of $\Mov(X)^{\vee}$. 

The cone $\Mov(X)^{\vee}$ has dimension $\rho_X$ in $\N(X)$, hence the classes $[C_{D_0}],[C_{D_1}],\dotsc,[C_{D_r}]$ generate all $\N(X)$. Since $\N(D_0,X)\subsetneq\N(X)$, and $[C_{D_0}]\in\N(D_0,X)$, there exists $i\in\{1,\dotsc,r\}$ such that $[C_{D_i}]\not\in\N(D_0,X)$. Suppose $i=1$ for simplicity.
\end{parg}
\begin{parg}\label{influenza}
We show that, up to replacing $C_{D_1}$ with $C_{D_i}$ for some $i\in\{2,\dotsc,r\}$, we can assume that $D_0\cap D_1\neq \emptyset$.

Indeed, suppose that  $D_0\cap D_1=\emptyset$; in particular we have $D_1\cdot C_{D_0}=0$.
We claim that there exists a one-dimensional face $\sigma$ of $\Mov(X)^{\vee}$ such that $\langle[C_{D_0}]\rangle+\sigma$ is a face of  $\Mov(X)^{\vee}$ and $D_1\cdot \sigma>0$.

This is an elementary fact from convex geometry, that can be seen as follows.
Let $\tau_1,\dotsc,\tau_s$ be all the one-dimensional faces of  $\Mov(X)^{\vee}$, except $\langle[C_{D_0}]\rangle$.
Notice that since $D_1$ is fixed, we have $[D_1]\not\in\Mov(X)=(\Mov(X)^{\vee})^{\vee}$, hence  the hyperplane $D_1^{\perp}$ intersects the cone $\Mov(X)^{\vee}$ in its interior, and contains $[C_{D_0}]$; in particular there exists some $i\in\{1,\dotsc,s\}$ such that $D_1\cdot\tau_i>0$.

Since $\langle[C_{D_0}]\rangle$ is a face of $\Mov(X)^{\vee}$, there exists
 $\alpha\in\Nu(X)$ such that $\alpha\cdot [C_{D_0}]=0$ and $\alpha\cdot\tau_i>0$ for every $i=1,\dotsc,s$. Set 
$$\lambda:=\max_{i=1,\dotsc,s}\frac{D_1\cdot\tau_i}{\alpha\cdot\tau_i}$$
and $I:=\{i\in\{1,\dotsc,s\}\,|\,(D_1\cdot\tau_i)/(\alpha\cdot\tau_i)=\lambda\}$.
Notice that $\lambda>0$, and that $D_1\cdot\tau_i>0$ for every $i\in I$.
Now set $\beta:=\lambda\alpha-[D_1]\in\Nu(X)$. Then $\beta\cdot [C_{D_0}]=0$, $$\beta\cdot\tau_i=\lambda\alpha\cdot\tau_i-D_1\cdot\tau_i\geq 0\quad\text{for every $i=1,\dotsc,s$,}$$ and $\beta\cdot\tau_i=0$ if and only if $i\in I$. Thus $\eta:=\beta^{\perp}\cap \Mov(X)^{\vee}$ is a face of $\Mov(X)^{\vee}$, of dimension at least $2$, containing  $\langle[C_{D_0}]\rangle$, and
such that $D_1\cdot\tau>0$ for every one-dimensional face $\tau$ of $\eta$ different from $\langle[C_{D_0}]\rangle$.
 Now we just choose a one-dimensional face $\sigma$ of $\eta$ such that $\langle[C_{D_0}]\rangle+\sigma$ is a $2$-dimensional face.
\end{parg}
\begin{parg}
Clearly $[C_{D_1}]\not\in\sigma$, because $D_1\cdot C_{D_1}<0$. Thus
by \ref{coop} there exists $i\in\{2,\dotsc,r\}$ such that $\sigma=\langle[C_{D_i}]\rangle$; for simplicity we assume $i=2$. 

Since $D_0\cap D_1=\emptyset$, we have $D_1\cdot C=0$ for every curve $C\subset D_0$, hence $\N(D_0,X)\subseteq D_1^{\perp}$. Since $\N(D_0,X)$ is a hyperplane, we get $\N(D_0,X)= D_1^{\perp}$. 
On the other hand $D_1\cdot C_{D_2}>0$,
therefore we have $[C_{D_2}]\not\in D_1^{\perp}=\N(D_0,X)$. 
Moreover
 $D_1$ and $D_2$ are not numerically proportional, hence $D_1^{\perp}\neq D_2^{\perp}$. This means that $D_2^{\perp}\neq\N(D_0,X)$, therefore $D_0\cap D_2\neq\emptyset$, and we can replace $D_1$ with $D_2$. Thus we can assume that $D_0\cap D_1\neq\emptyset$.
\end{parg}
\begin{parg}
We show that $D_1$ is of type $(3,2)$. 

Otherwise, 
let $X\dasharrow\w{X}\stackrel{g}{\to} Y$ be the MMP for $D_1$.
 Since $D_0$ cannot contain exceptional planes by Lemma \ref{pisa}(2), the transforms $\w{D}_0$ and $\w{D}_1$ respectively of $D_0$ and $D_1$ in $\w{X}$ must still intersect; recall also that  $\w{D}_1$ does not contain exceptional lines by Rem.~\ref{viapeyron}. 
Since  $\dim(\w{D}_0\cap\w{D}_1)=2$ and $g$ is either of type $(3,0)$ or of type $(3,1)$ by Th.~\ref{typefixed}, $g$ cannot be generically finite on  $\w{D}_0\cap\w{D}_1$.
Let $\Gamma\subset\w{D}_0\cap\w{D}_1$ be an irreducible curve, disjoint from any exceptional line, and contained in a fiber of $g$. 

Let $\Gamma'\subset X$ be the transform of $\Gamma$. Then $\Gamma'\subset D_0$ and $\Gamma'\equiv m C_{D_1}$ for some $m>0$, 
 contradicting that $[C_{D_1}]\not\in\N(D_0,X)$.
Therefore $D_1$ is of type $(3,2)$.
\end{parg}
\begin{parg}
Let $f\colon X\to X_1$ be the contraction of type $(3,2)$ with $\Exc(f)=D_1$ and $\NE(f)=\langle[C_{D_1}]\rangle$. We show that $f$ is of type $(3,2)^{sm}$, and that $X_1$ is smooth and Fano.

Since $D_0\cap D_1\neq\emptyset$, $D_0$ intersects some non-trivial fiber $F$ of $f$; on the other hand we must have $\dim (D_0\cap F)=0$, because $\N(F,X)=\R[C_{D_1}]$ and  $[C_{D_1}]\not\in\N(D_0,X)$. Therefore it cannot be 
$D_0\cdot C_{D_1}=0$ (for otherwise $D_0\supset F$), and we conclude that $D_0\cdot C_{D_1}>0$ and that  $D_0$ intersects every non-trivial fiber $F$ of $f$. This in turns shows that every such fiber has dimension one, and hence that $f$ is of type $(3,2)^{sm}$ and $X_1$ is smooth (see \ref{basic0}). Moreover $\NE(f)$ is the unique $D_1$-negative extremal ray of $\NE(X)$ (see Rem.~\ref{pisa}(3)), hence $X_1$ is Fano by \cite[Prop.~3.4(i)]{wisn}.
\end{parg}
\begin{parg}\label{old}
We show that  $[C_{D_0}+C_{D_1}]\in\mov(X)$.

We have already $D_0\cdot C_{D_1}>0$, and
by \cite[Lemma 2.2(b)]{cdue} we also have $D_1\cdot C_{D_0}>0$.
Let us consider $C_{D_0}+C_{D_1}$. We have $M\cdot (C_{D_0}+C_{D_1})\geq 0$ for every movable divisor $M$ of $X$, and $D\cdot(C_{D_0}+C_{D_1})\geq 0$ for every fixed prime divisor $D\subset X$ different from $D_1$ and $D_2$. Moreover
$$D_0\cdot (C_{D_0}+C_{D_1})=D_0\cdot C_{D_1}-1\geq 0,$$
and similarly for $D_1$. Recall that
$$\Eff(X)=\langle[D]\rangle_{D\text{\,fixed}}+\Mov(X)$$
(see \ref{otite}); thus we have $[C_{D_0}+C_{D_1}]\in\Eff(X)^{\vee}=\mov(X)$.
\end{parg}
\begin{parg}
By \ref{old}, the face $\langle[C_{D_0}],[C_{D_1}]\rangle$ of $\Mov(X)^{\vee}$ intersects $\mov(X)$ outside zero. Dually, this means that $\tau:=\Mov(X)\cap C_{D_0}^{\perp}\cap C_{D_1}^{\perp}$ is a codimension $2$ face of $\Mov(X)$, contained in the boundary of $\Eff(X)$. 

Let $\tau_0\in\MCD(X)$ be such that $\tau_0\subseteq\tau$ and $\dim\tau_0=\rho_X-2$, and let $g\colon X\dasharrow T$ be the rational contraction such that $g^*(\Nef(T))=\tau_0$. We have $\rho_T=\rho_X-2$, and since $\tau_0$ is contained in the boundary of $\Eff(X)$, $g$ is of fiber type.

Since $\tau_0\subset C_{D_1}^{\perp}$, by Lemma \ref{treno}(2) the composition $h:=g\circ f_1^{-1}\colon X_1\dasharrow T$ is a rational contraction, of fiber type;
moreover
 $\rho_T=\rho_{X_1}-1$. We conclude that $h$ is an elementary rational contraction of fiber type for $X_1$, therefore $\rho_{X_1}\leq 11$ by Th.~\ref{elemft}. Thus $\rho_{X_1}=11$, $\rho_{X}=12$,  
and we get $(i)$.
\end{parg}
\vspace{-20pt}
 \end{proof}
\subsection{Proof of Theorem \ref{onepoint}}
\label{last}
\noindent In this section we show the following result, which implies Th.~\ref{onepoint}.
\begin{thm}\label{finally}
Let $X$ be a smooth Fano $4$-fold with 
 $\rho_X\geq 12$. 
If $X$ has a fixed prime divisor  of type $(3,0)^{sm}$, then $\rho_X=12$ and there exists $f\colon X\to X_1$ 
where $X_1$ is smooth and Fano, $f$ is the blow-up of a smooth, irreducible surface,  and $X_1$ has an elementary rational contraction onto a $3$-fold.
\end{thm}
Before proving Th.~\ref{finally}, we need an auxiliary lemma.
\begin{lemma}\label{nofixed}
Let $X$ be a smooth Fano $4$-fold with $\rho_X\geq 7$, $D\subset X$ a fixed prime  divisor of type $(3,0)^{sm}$, and $X\dasharrow\w{X}\la Y$
 the MMP for $D$, 
where $Y$ is a smooth Fano $4$-fold and $f$ is the blow-up of the point $p\in Y$.
 Then $p$ cannot belong to any fixed prime divisor in $Y$.
\end{lemma}
\begin{proof}
By Lemma \ref{firstproperties}(2) and (3), $p$ cannot belong to a fixed prime divisor of type $(3,2)$, nor to an exceptional plane. 
Suppose that $p$ belongs to a fixed prime divisor $B\subset Y$. Let $Y\dasharrow\w{Y}\stackrel{g}{\to} Z$ be the MMP for $B$, $\w{B}\subset\w{Y}$ the transform of $B$,
 and $\tilde{p}\in\w{B}$ the image of $p$.

If $B$ is of type $(3,0)^{sm}$, let $C_1\subset B$ the transform of a general line through $\tilde{p}$ in $\Exc(g)\cong\pr^3$. Then $p\in C_1$ and 
$-K_Y\cdot C_1=3$, contradicting Lemma \ref{firstproperties}(2).
If $B$ is of type $(3,0)^{Q}$, let $C_2\subset B$ the transform of a general line through $\tilde{p}$ in $\Exc(g)\cong Q\subset\pr^4$.  Then $p\in C_2$,  
$-K_Y\cdot C_2=2$, and there are infinitely such curves, contradicting Lemma \ref{firstproperties}(3).
The case where $B$ is of type $(3,1)^{sm}$ is analogous.
\end{proof}
\begin{proof}[Proof of Th.~\ref{finally}]
Since $X$ has a fixed prime divisor of type $(3,0)^{sm}$, in particular $X$ is not a product of surfaces, and moreover $\rho_X\geq 12$; thus we have $\delta_X\leq 2$ by Th.~\ref{summary}.
If $\delta_X=2$, then \cite[Th.~1.2]{cdue} yields the statement.
Therefore we assume for the rest of the proof that
 $\delta_X\leq 1$. 
\begin{parg}
We know by \ref{otite} and Lemma \ref{macau} that $\Eff(X)$ is generated by the classes of fixed prime divisors. Moreover  Th.~\ref{everyfixeddiv} implies that the fixed prime divisors of $X$ cannot be all of type $(3,0)^{sm}$. Therefore there exist two fixed prime divisors $D,B_X\subset X$ such that $D$ is of type $(3,0)^{sm}$,
$B_X$ is not of type $(3,0)^{sm}$, and $\langle[D],[B_X]\rangle$ is a face of $\Eff(X)$. Moreover, again by Lemma \ref{macau}, we have  $\langle[D],[B_X]\rangle\cap\Mov(X)=\{0\}$.

Let $X\dasharrow\w{X}\stackrel{f}{\to} Y$ be the MMP for $D$, so that $Y$ is a smooth Fano $4$-fold with $\rho_Y\geq 11$, and $f$ is the blow-up of the point $p\in Y$.
\end{parg}
\begin{parg}\label{passaggio}
We note that $Y$ cannot have a  covering family of rational curves of anticanonical degree $3$. Indeed if so, the point $p$ should be contained in a connected curve of anticanonical degree $3$, which is impossible by Lemma \ref{3}.
\end{parg} 
\begin{parg}\label{secondo}
Let $C_1,\dotsc,C_s\subset Y$ be the images of all exceptional lines in $\w{X}$, and recall that $p\in C_i$ for every $i=1,\dotsc,s$.
We show that $[C_i]\in\mov(Y)$  for every $i=1,\dotsc,s$.

Indeed if  $[C_i]\not\in\mov(Y)$,
 there is a one-dimensional face $\sigma$ of $\Eff(Y)$ such that $\sigma\cdot C_i<0$. By Cor.~\ref{noquasiel}, $\sigma$ is not contained in $\Mov(Y)$, therefore $\sigma=\langle[G]\rangle$ where $G\subset Y$ is 
 a fixed prime divisor with $G\cdot C_i<0$. This gives $p\in G$, contradicting Lemma \ref{nofixed}.
\end{parg}
\begin{parg}\label{sabrina}
We show that $Y$ has a fixed divisor of type $(3,2)$.

By Lemma \ref{bijection}(3) the divisor $B_X\subset X$ is the transform of a fixed prime divisor $B_Y\subset Y$, and Lemma \ref{transform} yields that $B_Y$ is not of type $(3,0)^{sm}$, because $B_X$ is not. Therefore $B_Y$ is either of type $(3,2)$, or of type $(3,1)^{sm}$, or of type $(3,0)^Q$ (see Th.~\ref{typefixed}).

If $B_Y$ is of type $(3,2)$ we are done, and if $B_Y$ is of type $(3,1)^{sm}$, then $Y$ has also a fixed divisor of type $(3,2)$, by Prop.~\ref{(3,1)}.

Finally suppose that $B_Y$ is of type $(3,0)^Q$, and apply Prop.~\ref{(3,0)Q}. By \ref{passaggio}, we conclude that $Y$ has a prime divisor $T$ covered by rational curves of anticanonical degree $1$. For every $i=1,\dotsc,s$ we have $C_i\cap T=\emptyset$ by Lemma \ref{firstproperties}(2) and Lemma \ref{SQM}(1); in particular $T\cdot C_i=0$.

Since $T$ is effective, $\tau:=\mov(Y)\cap T^{\perp}$ is a face of $\mov(Y)$.
 On the other hand $[C_i]\in\mov(Y)$ by \ref{secondo}, thus $[C_i]\in\tau$ for all $i=1,\dotsc,s$.
By Lemma \ref{dimension} we get
$$\dim\tau\geq \dim\N(D,X)-1\geq\rho_X-2=\rho_Y-1,$$ 
thus $\dim \tau=\rho_Y-1$.
Dually this means that $[T]$ generates a 
 one-dimensional face of $\Eff(Y)$.
Cor.~\ref{noquasiel} implies that $T$ is a fixed divisor,
hence it is  of type $(3,2)$ by Lemma \ref{divisorcovered}.
\end{parg}
\begin{parg}
We show that $X$ has a fixed prime divisor $E_X$ of type $(3,2)$, disjoint from $D$. This is enough to conclude the proof, because then we have $\N(E_X,X)\subseteq D^{\perp}\subsetneq\N(X)$, and 
Prop.~\ref{zero} yields the statement. 

We have shown in \ref{sabrina} that $Y$ has 
a fixed prime divisor $E_Y$ of type $(3,2)$; let $E_X\subset X$ be its transform.
Set $\psi:=f\circ\ph\colon X\dasharrow Y$.
Since $E_Y$ is covered by curves of anticanonical degree $1$, by Lemma \ref{locus}(1) $E_Y$ is contained in the open subset where $\psi$ is an isomorphism. Hence  $D\cap E_X=\emptyset$, and moreover 
 $E_X$ is covered by curves of anticanonical degree $1$ and is not nef. Thus $E_X$ is
a fixed prime divisor of type $(3,2)$ by Lemma \ref{divisorcovered}.
\end{parg}
\vspace{-20pt}
\end{proof}
\section{Examples}\label{examples}
\begin{example}\label{example1}
Let $p_1,\dotsc,p_r\in\pr^4$ be general points, and set
$\w{X}_r:=\Bl_{p_1,\dotsc,p_r}\pr^4$.  
It follows from the work of Bauer \cite{bauer} (see also \cite[\S 12.5]{mukaibook}) and Mukai \cite{mukaiADE} that, for every $r\leq 8$, there is a SQM $\w{X}_r\dasharrow X_r$ such that $X_r$ is a smooth Fano $4$-fold, with $\rho_{X_r}=1+r\leq 9$, and every fixed prime divisor of $X_r$ is of type $(3,0)^{sm}$.
Furthermore,   for $r\geq 4$, $\Eff(X_r)$ is generated by fixed prime divisors.

Let us consider the case $r=8$.
The $4$-fold $\w{X}_8$ is studied in detail in \cite[\S 2]{mukaiADE}; let us briefly recall some of Mukai's results. 

The points $p_1,\dotsc,p_8\in\pr^4$ determine $8$ points $q_1,\dotsc,q_8\in\pr^2$, up to projective equivalence, via the classical association or Gale duality \cite[Ch.~III]{dolgort}. Consider the del Pezzo surface $S:=\Bl_{q_1,\dotsc,q_8}\pr^2$, and let $H\in\Pic S$ be the pull-back of $\ma{O}_{\pr^2}(1)$.

Then $\w{X}_8=\Bl_{p_1,\dotsc,p_8}\pr^4$ is isomorphic to 
the moduli space of rank $2$ torsion free sheaves $E$ on $S$, with $\det(E)=-K_S$ and $c_2(E)=2$, semistable with respect to $-K_S+2H$.
Moreover there exists a SQM $\ph\colon \w{X}_8\dasharrow X_8$ 
with $X_8$ smooth and Fano, and $\rho_{X_8}=9$; $X_8$ is isomorphic to 
the moduli space of rank $2$ torsion free sheaves $E$ on $S$, with $\det(E)=-K_S$ and $c_2(E)=2$, semistable with respect to $-K_S$.

The indeterminacy locus  $\w{X}_8\smallsetminus\dom\ph$ contains the following exceptional lines: the transforms of the $28$ lines in $\pr^4$ through $2$ of the blown-up points, and the transforms of the $8$ rational normal quartics through $7$ of the blown-up points. 
It is interesting to note that $h^0(X_8,-K_{X_8})=h^0(\w{X}_8,-K_{\w{X}_8})=6$, $(-K_{\w{X}_8})^4=-23$, and $(-K_{X_8})^4=13$ (these numbers can be computed using the results in section \ref{chi}). In particular $-K_{X_8}$ is not very ample, while $-K_{X_7}$ is (see \cite[\S 1.5]{parabolic}).
\end{example}
\begin{example}\label{example2}
Let $p\in\pr^4$ and let $\Lambda\subset\pr^4$ be a plane containing $p$. Let $Y\to\Bl_p\pr^4$ be the blow-up of the transform of $\Lambda$ in $\Bl_p\pr^4$. Then $Y$ is Fano, toric, with $\rho_Y=3$. The blow-up $\w{X}$ of $Y$ at two general points is still toric, and there is a SQM $\w{X}\dasharrow X$, given by a composition of $3$ flips, such that $X$ is a toric Fano $4$-fold with $\rho_X=5$; this is $R_3$ in Batyrev's list \cite{bat2}. One can check that $X$ has $6$ fixed prime divisors, of which only two are of type $(3,0)^{sm}$, the two exceptional divisors of $X\dasharrow Y$.
\end{example}

\providecommand{\noop}[1]{}
\providecommand{\bysame}{\leavevmode\hbox to3em{\hrulefill}\thinspace}
\providecommand{\MR}{\relax\ifhmode\unskip\space\fi MR }
\providecommand{\MRhref}[2]{%
  \href{http://www.ams.org/mathscinet-getitem?mr=#1}{#2}
}
\providecommand{\href}[2]{#2}

\end{document}